\documentclass[11pt,a4paper]{article}
\usepackage{CJK}
\usepackage{tabls}
\usepackage{bm}
\usepackage{multirow}
\usepackage{amssymb}
\usepackage{amsfonts}
\usepackage{mathrsfs}
\usepackage{empheq}
\usepackage{amsmath}
\usepackage{amsthm}
\usepackage{graphicx}
\usepackage{color}
\usepackage{indentfirst}
\usepackage{hyperref}
\usepackage{float}
\usepackage{cases}
\usepackage[titletoc]{appendix}

\usepackage[colorinlistoftodos,french]{todonotes}

\allowdisplaybreaks%

\def\f{\frac}
\def\pa{\partial}

\def\lab{\label}

\def\i1n{i=1,\cdots,n}
\def\j1n{j=1,\cdots,n}
\def\ij1n{i,j=1,\cdots,n}

\def\R{\mathbb R}
\def\N{\mathbb N}
\def\Z{\mathbb Z}
\def\C{\mathbb C}

\def \i{\mathrm i}

 \numberwithin{equation}{section}
\theoremstyle{definition}

 \newtheorem{thm}{\indent Theorem}[section]
 \newtheorem{cor}{\indent Corollary}[section]
 \newtheorem{lem}{\indent Lemma}[section]
 \newtheorem{prop}{\indent Proposition}[section]
 \newtheorem{defn}{\indent Definition}[section]
 \newtheorem{rem}{\indent Remark}[section]

 \newtheorem{property}{\indent Property}[section]

\theoremstyle{definition}

\newcommand\bna{\begin{eqnarray*}}
\newcommand\ena{\end{eqnarray*}}

\newcommand\bnan{\begin{eqnarray}}
\newcommand\enan{\end{eqnarray}}

\newcommand\bnp{\begin{proof}}
\newcommand\enp{\end{proof}}

\newcommand\bneq{\begin{eqnarray*}\left\lbrace \begin{array}{rcl}}
\newcommand\eneq{\end{array} \right.\end{eqnarray*}}
\newcommand\bneqn{\begin{eqnarray}\left\lbrace \begin{array}{rcl}}
\newcommand\eneqn{\end{array} \right.\end{eqnarray}}
\DeclareMathOperator{\differential}{Diff}

\newcommand{\be}{\begin{equation}}
\newcommand{\ee}{\end{equation}}
\newcommand{\beq}{\begin{equation*}}
\newcommand{\eeq}{\end{equation*}}

\newcommand{\Mx}{Q}

\newcommand{\Bo}{\ensuremath{\mathcal{B}}}

\begin{document}

\begin{CJK*}{GB}{gbsn}
\title{\bf On the Observability Inequality of Coupled Wave Equations: the Case without Boundary}

\author{Yan Cui \thanks{School of Mathematics (Zhuhai), Sun Yat-sen University, Zhuhai, 519082, P. R. China. E-mail: \texttt{cuiy27@mail.sysu.edu.cn}. }
, Camille Laurent \thanks{CNRS, UMR 7598, Laboratoire Jacques-Louis
Lions, F-75005, Paris, France and Sorbonne Universit\'e, Universit\'e Paris-Diderot SPC, CNRS,  Laboratoire Jacques-Louis Lions,  F-75005 Paris, E-mail:   \texttt{laurent@ann.jussieu.fr.} }
 and Zhiqiang Wang\thanks{School of Mathematical Sciences and Shanghai Key Laboratory for Contemporary Applied Mathematics,
Fudan University, Shanghai 200433, P. R. China.
E-mail: \texttt{wzq@fudan.edu.cn}. }
}

\date{}
\maketitle

\begin{abstract}
In this paper, we study the observability and controllability of wave equations coupled by
first or zero order terms on a compact manifold. We adopt  the approach in Dehman-Lebeau's   paper \cite{DehmanLebeau09} to prove that:
the weak observability inequality holds for wave equations  coupled by first order terms on compact manifold without boundary  if and only if a class of ordinary differential equations related to the symbol of the first order terms along the Hamiltonian flow are exactly controllable.
We also compute the higher order part of the observability constant and the observation time.
By duality, we obtain the  controllability of the dual control system in a finite co-dimensional space. This gives the full controllability under the assumption of unique continuation of eigenfunctions.
Moreover, these results can be applied to the systems of wave equations coupled by zero order terms of cascade structure after an appropriate change of unknowns and spaces. Finally, we provide some concrete examples as applications where the unique continuation property indeed holds.
\end{abstract}


 \section{Introduction and main results}\lab{section1.1}
\subsection{Coupling of order one}
Let ~$(\mathcal{M},g)$~ be a compact connected ~$n$-dimensional Riemannian manifold without boundary. Denote $\Delta_g$ the Laplace-Beltrami operator on $\mathcal{M}$ for the metric $g$. We consider the observability and control problem for the system of coupled wave equations:
\be
    \lab{mainwave1.1}
    \begin{cases}
        \pa_t^2V-\Delta_g V+ LV=0,\\
        (V(0),\partial_tV(0))=(V_0,V_1),
    \end{cases}
\ee
where~$V=(V^1,\cdots,V^N)^{tr}$ with $N\in \Z^+$ and $L$ is a matrix of differential operator of order one on $\R\times \mathcal{M}$ of the form
\be
    \lab{Ldiff}
            L=   A_0\pa_t + A_1,
\ee
with
$A_k\in \mathcal{C}^\infty(\R;\differential^k(\mathcal{M};\C^{N\times N})),  \ (k=0,1).$
Here $\differential^k(\mathcal{M};\C^{N\times N})$ is the set of matricial differential operators of order $k$ in space with smooth coefficients.

%

It is known that the weak solution of the Cauchy problem of System \eqref{mainwave1.1} exists for any initial data $(V_0,V_1)\in (H^1)^N\times (L^2)^N$ (see \cite{Pazy83}). Here and hereafter, $H^s\ (s\in \mathbb{R})$ denotes the Sobolev space  on manifold $\mathcal{M}$  with the norm defined as follows:
$\|f\|_{H^s}^2 = \| \Lambda^s f \|_{L^2}^2$,
where
\be\lab{main-lambda}\begin{split}
\Lambda^s f : = (-\Delta+1)^{\frac{s}{2}} f = \sum_{j \in \N} (\kappa_j + 1)^\frac{s}{2} (f, e_j)_{L^2} \, e_j , \quad  s\in \R
\end{split}
\ee
$(e_j)_{j \in \N}$ the eigenfunctions of the Laplace-Beltrami operator associated
to the eigenvalues  $(\kappa_j)_{j \in \N}$ which forms a Hilbert basis of $H^s$ .
In this context, we are interested in the following observability problem.

\begin{defn}
\label{defExactobs}
 We say that System \eqref{mainwave1.1} is {\bf Exactly Observable} on $[0,T]$, if the solutions of \eqref{mainwave1.1}  satisfy {\bf  Observability Inequality}
\be
    \lab{mainobsins}
    C^1_{obs} \int_0^T\|DV(t)\|^2_{(L^2)^K}dt \geq \|(V_0,V_1)\|^2_{(H^1)^N\times (L^2)^N},
\ee
where~$C^1_{obs}>0$
is a constant independent of the initial data ~$(V_0,V_1)$ and ~the observation operator $D\in \mathcal{C}^\infty(\R;\differential^1( \mathcal{M};\C^{K\times N}))$($\differential^1( \mathcal{M};\C^{K\times N})$ is a matrix of differential operator of order one on $\R\times \mathcal{M}$ taking the form)
\be
    \lab{Ddiff}
            D=   D_0\pa_t + D_1,
 \ee
 with
$D_k\in \mathcal{C}^\infty(\R;\differential^k(\mathcal{M};\C^{K\times N})),  \ (k=0,1).$

\end{defn}
\begin{defn}
We say that System \eqref{mainwave1.1} is {\bf Weakly Observable} on $[0,T]$, if the solutions of  \eqref{mainwave1.1}  satisfy {\bf  Weak Observability Inequality}
\be
    \lab{mainobsinw}
    C^2_{obs} \int_0^T\|DV(t)\|^2_{(L^2)^K}dt+c_1\|(V_0,V_1)\|^2_{(H^{\f{1}{2}})^N\times (H^{-\f{1}{2}})^N}\geq \|(V_0,V_1)\|^2_{(H^1)^N\times (L^2)^N},
\ee
where~$C_{obs}^2>0$~and $c_1$ are constants independent of the initial data ~$(V_0,V_1)$ and the observation operator $D$ is defined by \eqref{Ddiff}.
\end{defn}

Roughly speaking, the weak observability can be understood as the observability of functions with high frequency, that is,  $\|(V_0,V_1)\|_{(H^1)^N\times (L^2)^N}\gg \|(V_0,V_1)\|_{(H^{\f{1}{2}})^N\times (H^{-\f{1}{2}})^N}$.


\bigskip
 We mention a few notational conventions that we will use throughout.
We will use notation ~$\dot{X}=\f{dX}{dt}$~. We denote ~$a^{*}=\bar{a}^{tr}$ the adjoint matrix of $a$, $a^{tr}$ the transpose of $a$ and $L^*$~ the adjoint operator of ~$L$ for the $L^2$ (or $(L^2)^N$) scalar product inherited from the Riemannian structure.
We denote $S^*\mathcal{M}$ the cosphere bundle of $\mathcal{M}$. $\varphi_{ t}(\rho_0)$ is the Hamiltonian flow of ~$|\xi|_x $ initiated at $\rho_0$ defined by the formula
\be\lab{ini:varphi}
  \varphi_{ t}(\rho_0)=(x(t),\xi(t)),~~\varphi_0(\rho_0)=\rho_0.
  \ee	
We refer to Section \ref{section1.2} for more details.
Then we state our main results:
\begin{thm}
\lab{mainthm-w}
Solutions of System \eqref{mainwave1.1} satisfy weak observability inequality \eqref{mainobsinw} on $[0,T]$ if and only if for any $\rho_0\in S^*\mathcal{M}$,  the finite dimensional control system %
\be
    \lab{mainmainode}
\begin{cases}
        \displaystyle \dot{ X}(t)=\f{1}{2}a^{*}(t,\varphi_{ t}(\rho_0))X(t)+\f{1}{2}d^{*}(t,\varphi_{ t}(\rho_0))u(t),  \\
        X(0)=X_0 \in \C^N,
\end{cases}
 \text{with control $u \in L^2 (0,T; \mathbb{C}^K)$ }
\ee
 is exactly controllable on $[0,T]$.
  Here $X(t)=(X_1(t),\cdots,X_N(t))^{tr}\in \C^N$ is the state variable. The coefficients matrice $a$ and $d$ are defined by $a:=a_0-\f{a_1}{i|\xi|_x}$ and $d:=d_0-\f{d_1}{i|\xi|_x}$, where $a_k\in C^\infty(\R;S_{phg}^k(T^*\mathcal{M};\C^{N\times N})) \ (k=0,1)$ is the homogenous principal symbol of $A_k$ defined in \eqref{Ldiff} and $d_k\in C^\infty(\R;S_{phg}^k(T^*\mathcal{M};\C^{K\times N})) \ (k=0,1)$ is the homogeneous principal symbol of $D_k$ defined in \eqref{Ddiff}, respectively.
\end{thm}


As it is quite classical in control theory, see \cite[Theorem 4.1]{DuprezOlive18} for an abstract version, the previous result gives the observability result if some unique continuation property is fulfilled. Let us be more precise in the case of time invariant equations.

\begin{property}
\label{propUCPeign}Assume $A_0$ and $A_1$ are time invariant.
We say that a system satisfies the \textbf{Unique continuation of eigenfunctions} if the following property holds:\\
For any $\lambda\in \C$, any solution $V\in (H^1)^N$ of
\be
\label{eqnUCP}
    \begin{cases}
     -\Delta_g V+\lambda^2 V+ (\lambda A_0+A_1)V=0,\\
     \lambda D_0V+D_1 V=0,
    \end{cases}
\ee
is the zero solution $V\equiv 0$.
\end{property}
\begin{thm}\lab{mainthmstrong}
Assume that $A_0$ and $A_1$ are time invariant.
In the setting of Theorem \ref{mainthm-w}, the following two statements are equivalent:
\begin{enumerate}
\item System \eqref{mainwave1.1} is exactly observable according to Definition \ref{defExactobs}.
\item Property \ref{propUCPeign} is satisfied and for any $\rho_0\in S^*\mathcal{M}$, System \eqref{mainmainode} is exactly controllable.
\end{enumerate}
\end{thm}

Now, we will be more precise about the inequality we can obtain. In a similar way to Lebeau \cite{Lebeau96} for the stabilisation problem (see also Laurent-L\'eautaud \cite{LaurentLeautaud16} for scalar control and Klein \cite{Klein17} for systems of damped waves), it is possible to characterize the constant in the high frequency part of the weak observability estimate. Rougly speaking, we prove that the constant of the high frequency part can be exactly determined by the Gramian of the ODE System \eqref{mainmainode}. We will need more definition now.

We define the  Gramian matrix of System \eqref{mainmainode} by the formula
\be
\lab{GGGGG}
G_{\rho_0}(T)=\f{1}{4}\int_0^T R^*(0,t;\rho_0)d^{*}(t,\varphi_t(\rho_0))d(t,\varphi_t(\rho_0))R(0,t;\rho_0)dt
\ee
where $R(\cdot,\cdot;\cdot)$ is the resolvent of {\eqref{mainmainode}} (see \cite[Proposition 1.5]{Coron07} for definition ).
We can also define a constant
\be\lab{r1}
\begin{split}
\mathfrak{K}(T):=&\min\limits_{\rho_0\in S^*\mathcal{M},\beta\in \C^N,|\beta|=1}\left\{\beta^{*}G_{\rho_0}(T)\beta\right\}\\
=&\min\limits_{\rho_0\in S^*\mathcal{M}}\max\left\{s\in \R\left|\beta^{*}(G_{\rho_0}(T)-sId_{N\times N})\beta\geq 0, \forall \beta\in \C^N\right.\right\}\\
=&\min\limits_{\rho_0\in S^*\mathcal{M}}\min\left\{\lambda \in \R\left|\lambda \text{ is an eigenvalue of }G_{\rho_0}(T)\right.\right\}.
\end{split}
\ee
The equality of the different definitions comes from the symmetry and positivity of Hermitian matrix $G_{\rho_0}(T)$.
Note that $\mathfrak{K}(T)\geq 0$ and we have $\mathfrak{K}(T)> 0$ if and only if $G_{\rho_0}(T)>0$ (in the sense of symmetric matrices) which is equivalent to the controllability of System \eqref{mainmainode} (see \cite{Coron07}).

Moreover, it is very important to estimate the optimal constant of the observability inequality since it is closely related to the cost of optimal control of the dual system. The following Theorem precises Theorem \ref{mainthm-w} and states what is the optimal constant of the high regularity term in the Weak Observability Inequality.
\begin{thm}\lab{theoremoptimalcons}
If $T>T_{crit}:=\inf\limits_{T_0}\{T_0|\min\limits_{\rho_0\in S^*\mathcal{M}}\det\left(G_{\rho_0}(T_0)\right)>0\}$, then Weak Observability Inequality \eqref{mainobsinw} holds with $C^2_{obs}= \f{1}{2\mathfrak{K}(T)}$.
Reciprocally,  if
Weak Observability Inequality \eqref{mainobsinw} holds for all
solutions of System \eqref{mainwave1.1},  then we have ~$T>T_{crit}$~ and ~$C^2_{obs}\geq \f{1}{2\mathfrak{K}(T)}$, where $G_{\rho_0}(T_0)$ and $\mathfrak{K}(T)$ are defined by  \eqref{GGGGG} and \eqref{r1}, respectively.

\end{thm}


 \begin{rem}Theorem \ref{theoremoptimalcons} says that the observability constant $C^2_{obs}$ blows up like  $1/2\mathfrak{K}(T)$  as $T\rightarrow T_{crit}$.
\end{rem}
%
%

Next we introduce the adjoint system of \eqref{mainwave1.1}
\be
    \label{mainwave-dual}
        \begin{cases}
            \pa_t^2 U-\Delta_g U+L^*U=D^*F, \\
            (U(0),\partial_t U(0))=(U_0,U_1).
    \end{cases}
\ee
where
~$U=(U^1,\cdot,U^N)^{tr}$, ~$F=(f_1,\cdots,f_N)^{tr}\in L^2(0,T; (H^{-1})^K)$ is control function.
Clearly, the weak solution of the Cauchy problem of System \eqref{mainwave-dual} exists for any initial data $(U_0,U_1)\in (L^2)^N\times (H^{-1})^N$ and
 forces $F=(f_1,\cdots,f_K)^{tr}\in L^2(0,T; (L^{2})^K)$ (see \cite{Pazy83}).

 Thanks to Liu-Lu-Zhang \cite[Theorem 3.2]{XQX2005} (see also Duprez-Olive \cite{DuprezOlive18} for similar results for time independent systems), we obtain the following  corollary concerning Finite Co-dimensional Controllability of System \eqref{mainwave-dual}.
 \begin{cor}
\lab{maincorfc}
 Assume that System \eqref{mainmainode} is exactly controllable for any ~$\rho_0\in S^*\mathcal{M}$ on $[0,T]$. Then, there exists a finite dimensional subspace $H_{fin}$ and a finite co-dimensional subspace $H_{cofin}$ with $H_{fin} \oplus H_{cofin} = (L^2)^N\times (H^{-1})^N$, such that:
for any initial data  $(U_0,U_1)\in  H_{cofin}$, there exists control $F\in L^2(0,T; (L^{2})^K)$ such that the solution of \eqref{mainwave-dual} satisfies $(U(T),\partial_tU(T))=(0,0)$.
    \end{cor}

%
Now, we want to give more qualitative properties of the HUM operator.  We will need to consider the change of variable corresponding to the half wave decomposition. More precisely, define $\Sigma (V_0,V_1)=(V_+,V_-)= \left(i\Lambda V_0+V_1,-i\Lambda V_0+V_1\right)$ with $\Lambda=(-\Delta_g+1)^{1/2}$, see Section \ref{subsectGramian} for more precisions.
Denote $\tilde{\mathcal{G}}_T:=G_T+\tilde{\mathcal{R}}_T$ the Gramian operator which is defined below by \eqref{obs-2.0intro}. If $\tilde{\mathcal{G}}_T$  is invertible, then define $\mathcal{L}_T=(\tilde{\mathcal{G}}_T)^{-1}$ the HUM (Hilbert Uniqueness Method) control operator.
As a byproduct of the proof of Theorem \ref{mainthm-w}, we obtain the following interesting characterization of $\mathcal{L}_T$ as a matricial pseudodifferential operator. This generalizes some results of Dehman-Lebeau \cite{DehmanLebeau09} in the scalar case to systems. Note that it is also related to some trivialization along the flow that are described in Burq-Lebeau \cite{BurqLebeau01} in the case with boundary.

\begin{thm}
\label{thmHUMintro}
Let  $V_*:=(V_0,V_1)\in (H^1)^N\times (L^2)^N$ be the  initial data of System \eqref{mainwave1.1}. Let $T_0>0$.  Then for any $T\in (0,T_0]$, we have
\be
    \lab{obs-2.0intro}
    \int_0^T\|DV(t)\|_{(L^2)^K}^2dt=((G_T+\tilde{\mathcal{R}}_T)\Sigma V_*,\Sigma V_*)_{(L^2)^{2N} },
\ee
 where ~$G_T\in \mathcal{C}^\infty(0,T_0;\Psi_{phg}^0(\mathcal{M};\C^{2N\times 2N}))$ and  $ \tilde{\mathcal{R}}_{T}\in \mathcal{B}(0,T_0;\mathcal{L}((H^{\sigma})^{2N},(H^{\sigma+1})^{2N}))$ is in a class of regularizing operators of order at least one.
Moreover,  the  principal symbol of $G_{T}$ can be characterized as follows:
\be\displaystyle
 \begin{split}
    \sigma_0(G_T)(\rho_{0})&=\left(
                    \begin{array}{cc}
                  G_{\rho_0}^+(T)  & 0 \\
                      0 & G_{\rho_0}^-(T)   \\
                    \end{array}
                  \right)\in \mathcal{C}^\infty(0,T_0;S_{phg}^0(\mathcal{M};\C^{2N\times 2N})),
 \end{split}
 \ee
 where $ G_{\rho_0}^{\pm}(T)$ are the Gramian matrices of the control systems
\be
    \lab{intromainodepm}
\left\{ \begin{array}{lll}
        \displaystyle \dot{ X}(t)=\f{1}{2}a^{*}_{\pm}(t,\varphi_{ \mp t}(\rho_0))X(t)+\f{1}{2}d^{*}_{\pm}(t,\varphi_{ \mp t}(\rho_0))u(t),  \\
        X(0)=X_0\in \C^N,
    \end{array}\right.
\ee
where $X(t)=(X_1,\cdots,X_N)^{tr}$ is a vector having $N$ components, $a_{\pm}=a_0{\pm}\f{a_1}{i|\xi|_x}$, $d_{\pm}=d_0{\pm}\f{d_1}{i|\xi|_x}$, $\varphi_{t}(\rho_0)$ is the Hamiltonian flow of ~$|\xi|_x $ initiated at $\rho_0$ and $u(t)\in L^2(0,T;\C^{K})$ is the control.

\end{thm}

The interest of this theorem is that at high frequency, the HUM operator $\mathcal{L}_T$ is a pseudolocal operator. That means that if one needs to control the initial data with a lot of oscillations localized only in some region of the phase space, the corresponding optimal HUM control will also present these oscillations only in the same region. We refer to the interesting numerical study of this fact in {\cite{LebeauNodet10}}  the scalar case where this property is explored. Note that this would be interesting to make similar numerical study in the vectorial case we consider.

\subsection{Coupling of order zero}
\label{sectzeroorder}
The purpose of this Section is to transfer the results we have obtained for coupling of order one to coupling of order zero. The main difference is that zero order coupling are not strong enough to transfer the information from a component to another in the natural spaces. Indeed, if we apply directly the results of the previous Section for zero order coupling, the coupling (considered as an operator of order one) will have zero principal symbol and thus, there will be no coupling at this level of regularity. So, we have to adapt the setting.

Before getting to a general result, let us study a first enlightening example: a system of two equations with cascade coupling that was  completely studied in Dehman-Le Rousseau-L\'eautaud \cite{DehmanLeLeautaud14}:
\be
    \begin{cases}
     \partial_t^2u-\Delta_g u=1_{\omega}g,\\
    \partial_t^2v-\Delta_g v+a(x)u=0.
    \end{cases}
\ee
It is clearly not possible to control both components in $H^1\times L^2$ with a control $g$ in $L^2(0,T;L^2))$, which is the natural regularity for scalar control. Indeed, if the initial conditions are zero for $u$ and $v$ and $g\in L^2(0,T;L^2))$, this will create some solutions $u$ in $C([0,T];H^1)$ and the source term $a(x)u$ (for $v$) will be in $C([0,T];H^1)$, which will create a solution $v$ in $C([0,T];H^2)$. So, in that case, the natural space of control is $H^1\times L^2$ for $u$  and $H^2\times H^1$ for $v$.
Then, we see that it is necessary to classify each variable of the system according to algebraic properties of the coupling and the control operator.

Now let us move to the $N\times N$ system
\be
    \label{wavezerointro}
        \begin{cases}
            \pa_t^2 U-\Delta_g U+AU=BG, \\
            (U(0), \partial_tU(0))=(U_0,U_1).
    \end{cases}
    \text{with control\ } G \in L^2(0,T;(L^2)^K)
\ee
where $A(x)$ is a matrix in $\R^{N\times N}$ and  $B(x)$ is a matrix in $\R^{N\times K}$.
Without loss of generality, we can assume $A(x)$ is  a matrix "subdiagonal by block" in $\R^{N\times N}$
\be \small
\label{structA}
A(x)=\begin{bmatrix}
    A_{11} & \dots  &\dots  & A_{1k} \\
    A_{21} & \dots  &\dots  & A_{2k} \\
    \vdots &  \ddots &\ddots  & \vdots \\
    0 & 0& A_{k,k-1}  & A_{kk}
\end{bmatrix}
\ee
with $A_{i,j} \in \R^{d_i \times d_j} \ (i=1,\cdots,k)$ and  $B(x)$ is a matrix in $\R^{N\times K}$ of the form
\be \small
\label{structB}
B(x)=[B_{11},0,\cdots,0]^{tr}
\ee
where $B_{11}\in \R^{d_1\times K}$.
In fact, any $A(x), B(x)$ can be transformed into these forms simultaneously by using one algorithm detailed in Section \ref{sectzero}.
 Noting the coupling of  structure by blocks, one can  analyze  the regularity of components in blocks and easily find out the natural space for solutions of \eqref{wavezerointro} is  $\mathcal{H}^s$ as follows. We have
$U\in \mathcal{H}^s$ if for every $i=1,...,k$, we have $U^i \in (H^{s+i-1})^{d_i}$ where $d_i$ is the dimension of $A_{i,i}$. That is
\be\lab{zero-Hs-spaceintro}
\mathcal{H}^s= (H^{s})^{d_1}\times (H^{s+1})^{d_2}\times \cdots \times (H^{s+k-1})^{d_k}.
\ee
The natural energy space is then $\mathcal{E}=\mathcal{H}^1\times \mathcal{H}^0$ and it appears that the important terms are the subdiagonal terms of $A$  which leads to define
\bnan \small
\label{Asub}
A_{sub}(x)=\begin{bmatrix}
    0 &  \dots  &\dots  &0 \\
    A_{21} & \dots  &\dots  & 0 \\
    \vdots &  \ddots &\ddots  & \vdots \\
    0 &  0  &A_{k,k-1}  & 0
\end{bmatrix}.
\enan
This gives the following theorem of control.
\begin{thm}\lab{thmcontrolmulti}
Assume that $A(x)$, $A_{sub}(x)$ and $B(x)$ have some decomposition as in \eqref{structA}, \eqref{Asub} and \eqref{structB}. The System \eqref{wavezerointro} is controllable in $\mathcal{E}=\mathcal{H}^1\times \mathcal{H}^0$ on $[0,T]$ with control $G\in L^2(0,T; (L^2)^K)$ if and only if we have the following two properties
\begin{itemize}
\item The  control system
 \be
    \lab{mainmainode1}
\left\{ \begin{array}{lll}
        \displaystyle \dot{ X}(t)=\f{1}{2}A_{sub}(\varphi_{ t}(\rho_0))X(t)+\f{1}{2}{B(\varphi_{ t}(\rho_0))u},  \\
        X(0)=X_0 \in \mathbb{R}^N.
    \end{array}\right.
\ee
with  control $u\in L^2(0,T; \mathbb{R}^K))$   is exactly controllable on $[0,T]$.
\item Unique continuation of eigenfunctions:

For any $\lambda\in \C$, any solution $V\in (H^1)^N$ of
\be
    \begin{cases}
     -\Delta_g V+ A^{*}(x)V=\lambda V,\\
     B^{*}V=0,
    \end{cases}
\ee
is $V=0$.
\end{itemize}
\end{thm}
{\color{red}
The equivalence is true once the polarized space $\mathcal{E}$ has been chosen and the decomposition in block has been specified. Theorem \ref{thmcontrolmulti} yields a necessary and sufficient condition once we fix the decomposition as in \eqref{structA} and \eqref{structB}. This decomposition might not be unique, but each decomposition, eventually after some change of unknown, gives a different result of control, positive or negative, which has its own interest. 

In Section \ref{sectiongettindiag}, we will show how the assumption of subdiagonal form is actually quite general. Indeed, for any couple $A(x)$, $B(x)$, there exist some change of unknown that lead the control problem to have the subdiagonal form expected. Yet, this is not unique. For instance, the trivial decomposition with only one block always works. In that case, $A_{sub}(x)=0$ and there is no coupling. Our result gives a necessary and sufficient condition for the control in $(H^1)^N\times (L^2)^N$. Yet, it is possible in some situation that another choice of decomposition would use better the coupling but at the cost of a loss in the space.
}

In the constant case $A(x)=A\in \R^{N\times N}$, $B(x)=B\in \R^{N\times K}$, the Brunovsky normal form (written in a slightly different way, see Proposition \ref{propBrunovsky}) always allows to put our control system in the expected subdiagonal form with the good property. In that case, it seems to be the optimal choice that gives the best controllability result. We obtain the following theorem.
\begin{thm}
\lab{cons-case-propintro}
Let $A,B$ constant satisfying the Kalman rank condition and $\omega$ satisfies Geometric Control Condition (\cite{BardosLebeauRauch92}). Then, there exists some integer $k\leq N$ and some $d_{i}\in \N$, $i=1,\cdots, k$, allowing to define the space $\mathcal{E}:=\mathcal{H}^1\times \mathcal{H}^0$ as in \eqref{zero-Hs-spaceintro} and some matrix $\Mx\in GL_{N}(\R)$
so that the System \eqref{wavezerointro} is controllable in $\Mx\mathcal{E}$ with control $G\in L^2(0,T; (L^2)^K)$.
\end{thm}
The matrix $\Mx$ and the integers $k$ and $d_{i}$ are strongly related to the Brunovsky normal form. Roughly speaking, this decomposition transforms the control problem in the control system with integrators
\bnan
\label{controlintegrator}
y_{1}^{(\alpha_{1})}=u_{1},\quad  \cdots\quad  ,y_{m}^{(\alpha_{m})}=u_{m}, \qquad\alpha_1,\cdots,\alpha_m\in \Z^+,
\enan
the state being $y_{1},y_{1}^{(1)},\cdots,y_{1}^{(\alpha_{1}-1)},\cdots, y_{m},y_{m}^{(1)},\cdots,y_{m}^{(\alpha_{m}-1)}$ and the controls being the $u_{i}$. In that setting, $k$ is $\max_{i=1,\cdots m}\alpha_{i}$, that is the stronger integrators. It is then natural that for the wave equation, the observations holds in some space $H^{k}$, that is we have integrated $k$ times from $H^{1}$ thanks to the regularization of the wave operator with respect to a source term. Note also that it is not clear that the space $\Mx\mathcal{E}$ is invariant by the equation, so we should precise which kind of control we mean (control to zero, from zero...). Yet, it will be a byproduct of the proof that $\Mx\mathcal{E}$ is invariant by the equation. So, here, by controllability in $\Mx\mathcal{E}$, we mean that any state in $\Mx\mathcal{E}$ can be controlled to a state in $\Mx\mathcal{E}$.

Some previous articles (Liard-Lissy \cite{LiardLissy17}, Lissy-Zuazua \cite{LissyZuazua}) already obtained some controllability property in this framework under the Kalman rank condition (in a more abstract and general setting). So, an improvement of our Theorem comes from the space where the controllability holds. We refer to Section \ref{sectioncouplingconstant} for more precisions.

\subsection{Other applications}
\subsubsection{Other equations: Parabolic and Schr\"odinger-like systems}
Thanks to the transmutation techniques, see for instance \cite{EZ11,Miller06}, all the results stated in this article might give results for the analog parabolic system and for systems of Schr\"odinger equations.

A lot of controllability results of parabolic system have been established and it would be impossible to give a complete view of the subject. We refer for instance to the survey paper Ammar-Khodja-Benabdallah-Gonz\'alez-Burgos-de Teresa \cite{AKBGBdTs:11}.
Under the assumption that the control domain and the coupling domain intersect each other, controllability results can be obtained under some algebraic conditions, like of Cascade type or Kalman Rank Condition \cite{AKBDGB2:09,AKBDK:05,Gue:07} (see also \cite{DL:16}). Note also that these papers about parabolic equations often contain as a byproduct some results of unique continuation for eigenfunctions that are in the assumptions of our theorem. In the opposite direction, we also would like to refer to the interesting paper of Boyer-Olive \cite{BO14} that gives several 1D counterexamples of unique continuation of eigenfunctions. It would be interesting to check if there is a link between this counterexamples and our assumption of controllability of the ODE problem. Are there some cases were the unique continuation is false while the weak observability is true? or backward?

 Using the transmutation method and removing the assumption of intersection of the domains of coupling and control, \cite{AlabauLeautaud13} obtains indirect controllability of parabolic system of cascade and symmetric under Geometric Control Conditions (GCC).

We refer to \cite{LMAL:16,LiardLissy17} for internal controllability results of systems of Schr\"odinger equations coupled with constant  zero order terms with good algebraic structure.

\subsubsection{The boundary case}
The choice of using Egorov Theorem for proving our results has the advantage to be simpler and more precise. Indeed, we get a structure of the HUM control operator and the exact constant of high frequency. 
However, it has the disadvantage that it does not apply (at least up to our knowledge) to the case of domains with boundary. Most of the results presented in this paper (with the notable exception of the description of the HUM control operator as a pseudo-differential operator) might remain true in the case of boundary. Yet, it requires different techniques. We are therefore planning  to prove similar result in the case of boundary in a forthcoming paper \cite{CuiLaurentboundary}. The proofs will be based on the full description of microlocal defect measures of sequences of solution of wave equations as performed in \cite{BurqLebeau01}.

\subsection{Previous results}
Let us  discuss briefly the previous work on controllability and observability problem for wave equations.
Russell \cite{Russell78} and  Lions \cite{Lions88} set up the duality and proved that the exact controllability of the control system can be equivalently reduced to the observability inequality for solutions of the adjoint system. Then Bardos-Lebeau-Rauch pointed out the Geometric Control Condition (GCC) is crucial to the controllability and stabilization of (scalar) wave equations \cite{BardosLebeauRauch88a,BardosLebeauRauch88b,BardosLebeauRauch92}. Note that in the framework of our Theorems the GCC for the control of the scalar wave equation $(\partial^2_t-\Delta_g) u=\chi_{\omega}(x)h$ is described in an equivalent way as the controllability for any $\rho_0\in S^*\mathcal{M}$ of the scalar control system $\dot{x}(t)=\chi_{\omega}(\varphi_t(\rho_0))u(t)$ with control $u$ on $[0,T]$.

Alabau-Boussouira \cite{Alabau03} first studied the indirect controllability of two wave equations with constant coefficients coupled by displacements via one boundary control. The controllability result was established in a multi-level energy space similar to \eqref{zero-Hs-spaceintro} and it was generalized to variable coefficients coupling  under geometric control conditions on coupling and  control domains in Alabau-Boussouira-L\'eautaud \cite{AlabauLeautaud13}. Other results for the related problem of stabilization were also formulated by the same authors \cite{Alabau02,AlabauLeautaud12} and then by Aloui-Daoulatli \cite{AlouiDaoulatli16}.

 In \cite{DehmanLeLeautaud14}, Dehman-Le Roussau-L\'{e}autaud   proved the controllability of two wave equations coupled by zero order terms of Cascade type on a compact manifold. Moreover, they gave  the sharp controllability time and a microlocal characterization of the HUM control operator similar to the one of Theorem \ref{thmHUMintro}. In Section \ref{sectionexample}, we explain how our main result allows to recover some results in \cite{DehmanLeLeautaud14} with the study of an appropriate ODE problem. The multi-speed case was also studied in \cite{DehmanLeLeautaud14}.  As stated earlier, under Kalman Rank Condition and GCC, exact controllability of systems of wave equations with  constant coefficients and general coupling structure of zero order were proved in \cite{LiardLissy17,LissyZuazua}. Note also the recent article of Alabau-Boussouira-Coron-Olive \cite{AlabauCoronOlive17} for 1-D hyperbolic systems where appears a condition on ODE problems related to ours and the use of its Gramian.

 There are  many other control problems which are closely related to  or strongly motivated by the study of controllability of systems of wave equations via less controls, for instance, the synchronization problems \cite{LiRao13,LiRao14},  desensitizing control problems \cite{Dager06,Lions88,Tebo08a,Alabau14insent} and simultaneous control problems \cite{AvdoninTucsnark01,Lions88}.

 Let us mention that the above results concern only the systems with coupling of order zero. As for the systems of wave equations coupled by first order terms, we refer to \cite{AWY17,CuiWang16} for the stability of such systems coupled by velocities under strong geometric conditions.

Recently, Klein \cite{Klein17} obtained some results related to ours for the stabilisation of wave equations. He computes the best exponent for the stabilization of wave equations on compact manifolds. The coefficient he obtains is therefore solution of some ODE system of matrices. In this context, an improvement of our paper is to recognize the relation between this coefficient  and the Gramian control operator of the ODE system.

\subsection{Plan of the paper}
The plan of the paper is the following.
 Section \ref{section1.2} is devoted to provide some preliminary works.  In Section \ref{subsectionEgorov} we give a proof of a System Egorov Theorem. In Section \ref{subsectionGarding} we recall the $N \times N$ Sharp G{\aa}rding Inequality useful in our context.
In  Section  \ref{sectionproof}, we get back to the control problem. In Section \ref{subsectGramian}, we provide a characterization of the principal symbol of
the Gramian operator. Our main results are proved in Section \ref{subsectionproofThm} and \ref{subsectproofthmbis}.
Section \ref{sectzero} is about the implications of our theorem (which concerned coupling by coefficients of order $1$) in the case of coupling by zero order coefficients.
Section \ref{sectionexample} is about examples of applications of our theorem, namely a cascade system,  an antisymmetric system and  a system coupling with constant coefficients.
We gathered in an Appendix several known results that we use through the article, namely the wellposedness of hyperbolic system on a compact manifold and some theorems of controllability of ODE and Finite Co-dimensional Controllability.

\medskip

\textbf{Acknowledgement:} The first author is partially supported by the fund of the Shanghai Key Laboratory for Contemporary Applied Mathematics in FDU(No. 74120-42080001).
 The second author would like to thank Matthieu L\'eautaud for several discussions.
The second author is partially supported by the Agence Nationale de la Recherche under grant ISDEEC ANR-16-CE40-0013. Part of this work was done when the second author was invited by Fudan University. He would like to warmly thank this institution for its hospitality. The third author was partially supported by the National Science Foundation of China (No. 11271082),
the State Key Program of National Natural Science Foundation of China (No. 11331004).

\section{Preliminary works}\lab{section1.2}
In Subsection \ref{subsectionEgorov} and Subsection \ref{subsectionGarding}, we prove System Egorov Theorem and $N\times N$ Sharp G{\aa}rding Inequality on manifold respectively.

\subsection{System Egorov Theorem}\lab{subsectionEgorov}
We consider the following hyperbolic system:
\be \lab{hyperbolic}\begin{cases}
    \pa_t U(t)-iH(t) U(t)=0,\\ U(s)=U_0,
\end{cases}
\ee
where
\be\lab{ht}
H(t)=c\Lambda Id_{N\times N}+iW_0(t),
\ee
$c\in \R$ and  $W_0(t)\in \mathcal{C}^\infty(0,T;\Psi_{phg}^0(\mathcal{M};\C^{N\times N}))$ is a matrix pseudodifferential operator of order 0. Denote $w_0=\sigma_0(W_0)\in \mathcal{C}^\infty(0,T;S_{phg}^0(T^*\mathcal{M};\C^{N\times N}))$  the principal symbol of $W_0$.

We define the notation $S(t,s)$~as the solution operator associated to \eqref{hyperbolic}, that is  ~$S(t,s)U_0=U(t)$.
The main result of this section is the following variation of Egorov Theorem  (see  \cite[Section A.1]{LaurentLeautaud16} for  its scalar  case).

\begin{thm}[System Egorov Theorem]
\lab{system egorov}
 For any
  $P_m(\cdot)\in \mathcal{C}^\infty(0,T;\Psi_{phg}^m(\mathcal{M};\C^{N\times N}))$~, ~$m\in \R$~,~$p_m(s,\cdot)=\sigma_m(P_m(s)),$
 there exist
 $Q(t,s)\in \mathcal{C}^\infty((0,T)^2;\Psi_{phg}^m(\mathcal{M};\C^{N\times N}))$,
$R(t,s)\in \mathcal{B}((0,T)^2,\mathcal{L}((H^\sigma)^N,(H^{\sigma+1-m})^N))$   and
$  \pa_t R(t,s),\pa_s R(t,s)\in \mathcal{B}((0,T)^2,\mathcal{L}((H^\sigma)^N,(H^{\sigma-m})^N))$
for any $\sigma \in \R$,
such that
\be
    S(s,t)^*P_m(s) S(s,t)-Q(t,s)=R(t,s), \quad (t,s)\in (0,T)^2.
\ee
 Moreover, the principal symbol of ~$Q(t,s)$~ is given by $q(t,s,\cdot)$ which satisfies:
\be
    \label{q}
    q(t,s,\rho)=R_1^{*}(t,s;\chi^c_{s,t}(\rho))p_m(s,\chi^c_{s,t}(\rho))R_1(t,s;\chi_{s,t}^c(\rho)),
\ee
where~$\chi^c_{t,s}(\rho)$~ is given by the flow of Hamiltonian vector field associated with ~$-c\lambda$
\be
\label{eqnHamiltonc}
    \f{d}{dt} \chi^c_{t,s}=H_{-c\lambda}(\chi^c_{t,s}),\quad \chi^c_{s,s}(\rho)=\rho\in T^*\mathcal{M}\setminus \{0\}.
\ee
and~$R_1(\tau,s;\rho)$~ satisfies
\be
\label{formulaR}
 \f{d R_1(\tau,s;\rho)}{d\tau}=R_1(\tau,s;\rho)w_0(\tau,\chi^c_{\tau,s}(\rho)); \quad
        R_1(s,s;\rho)=Id_{N\times N}.
\ee 
\end{thm}
Note that in fact, $\chi^c_{t,s}(\rho)=\chi^c_{t-s,0}(\rho)$ and the implicit formula \eqref{formulaR} defines well $R_{1}
$. We recall that $\lambda$  is defined in \eqref{main-lambda}.

The proof is inspired from \cite{LaurentLeautaud16} in the scalar case.
\begin{proof}
We firstly note that
~$S(t,s)$~ satisfies
\be\lab{seq}
    \pa_t S(t,s)-iH(t) S(t,s)=0,\quad S(s,s)=Id_{N\times N}.
\ee
where the time derivative is not to be taken as a derivative in a Banach space $\mathcal{L}((H^\sigma)^N,(H^{\sigma})^N)$ but in the weak sense, that is the derivative when the operator is applied to a fixed function. See for instance \cite[Corollary A.2]{LaurentLeautaud16} for more details.
 Since ~$S(t,s)S(s,t)=Id_{N\times N}$~, we have
\be\lab{adjointequations} \begin{split}
    \pa_t S(s,t)+i S(s,t) H(t)=0,\\
    \pa_t S(t,s)^*+i S(t,s)^*H(t)^*=0,\\
    \pa_tS(s,t)^*-iH(t)^*S(s,t)^*=0.
\end{split}
\ee
with $H^*(t)=\Lambda Id_{N\times N}-iW_0^*(t)$. The well-posedness of \eqref{hyperbolic} yields the following regularity properties
$S(t,s)\in \mathcal{B}((0,T)^2;\mathcal{L}((H^\sigma)^N),\pa_tS(t,s),\pa_s S(t,s)\in\mathcal{B}((0,T)^2;\mathcal{L}((H^\sigma)^N;(H^{\sigma-1})^N)$
 for all ~$\sigma\in \R$~, as well as for ~$S(t,s)^*$.

Now, setting
\be
    P(t,s)=S(s,t)^*P_m(s) S(s,t),
\ee
and using the above equations, we have $P(s,s)=P_m(s)$ with
\beq
    \label{ode}
    \pa_tP(t,s)=iH(t)^*P(t,s)-iP(t,s)H(t)=ic[\Lambda Id_{N\times N},P(t,s)]+W_0(t)^*P(t,s)+P(t,s)W_0(t).
\eeq
Here $[\cdot,\cdot]$ stands for the classic commutator.
We now construct an approximate pseudodifferential solution ~$Q(t,s)$. Its principal symbol $q(t,s,x,\xi)$ should satisfy
\be
    \label{commutor}
    \pa_t q(t,s,\cdot)=c\{\lambda,q(t,s,\cdot)\}+w_0^{*}(t,\cdot)q(t,s,\cdot)+q(t,s,\cdot)w_0(t,\cdot),~~q(s,s,\cdot)=p_m(s,\cdot),
\ee
 where~$\{.,.\}$~stands for the Poisson bracket in the $(x,\xi)$ variables.

We claim that $q(t,s,\rho)$ defined in \eqref{q} satisfies \eqref{commutor}.

Indeed, since $ \chi^c_{s,t}\circ\chi^c_{t,s}(\rho)=\rho$, we have $q(t,s,\chi^c_{t,s}(\rho))=R_1^*(t,s;\rho)p_m(s,\rho)R_1(t,s;\rho)$. So
\be\begin{split}
\label{formuleEgorovsymb}
 \frac{d}{dt}\left[q(t,s,\chi^c_{t,s}(\rho)) \right]         =& w_0^*(t,\chi^c_{t,s}(\rho))R_1^*(t,s;\rho)p_m(s,\rho) R_1(t,s;\rho)\\
       &+R_1^*(t,s;\rho)p_m(s,\rho) R_1(t,s;\rho)w_0(t,\chi^c_{t,s}(\rho))\\
        =&w_0^*(t,\chi^c_{t,s}(\rho))q(t,s,\chi^c_{t,s}(\rho))+q(t,s,\chi^c_{t,s}(\rho)) w_0(t,\chi^c_{t,s}(\rho)).
\end{split}
\ee
So, denoting $\widetilde{q}(t,s,\rho) =q(t,s,\chi^c_{t,s}(\rho)) $, \eqref{formuleEgorovsymb} can be written
\be\begin{split}
 \frac{d}{dt}\left[\widetilde{q}(t,s,\rho) \right]=&w_0^*(t,\chi^c_{t,s}(\rho))\widetilde{q}(t,s,\rho)+\widetilde{q}(t,s,\rho)w_0(t,\chi^c_{t,s}(\rho)).
\end{split}
\ee
Therefore, since \eqref{eqnHamiltonc} gives $\frac{d}{dt}\left[\widetilde{q}(t,s,\rho) \right] =(\pa_t q)(t,s,\chi^c_{t,s}(\rho))-c\{\lambda,q\}(t,s,\chi^c_{t,s}(\rho))$, we obtain for any ~$(t,s)\in(0,T)^2$~ and ~$\rho\in T^*\mathcal{M}$
$$\pa_t q(t,s,\chi^c_{t,s}(\rho))=w_0^*(t,\chi^c_{t,s}(\rho))q(t,s,\chi^c_{t,s}(\rho))+q(t,s,\chi^c_{t,s}(\rho)) w_0(t,\chi^c_{t,s}(\rho))+c\{\lambda,q\}(t,s,\chi^c_{t,s}(\rho)).$$
Since for fixed $(t,s)\in(0,T)^2$, $\chi^c_{t,s}$ is a bijection of $T^*\mathcal{M}$, it gives
$$\pa_t q(t,s,\rho)=w_0^*(t,\rho)q(t,s,\rho)+q(t,s,\rho) w_0(t,\rho)+c\{\lambda,q\}(t,s,\rho).$$
 We thus see that our definition \eqref{q} of ~$q(t,s,\rho)$~satisfies \eqref{commutor}.

The homogeneity of ~$\lambda$~ of order one allows to keep the homogeneity of $q(t,s,\rho)$.
This allows to select one $Q(t,s)$, so that
\be
    \label{Qq}
    Q(t,s)\in \mathcal{C}^\infty((0,T)^2;\Psi_{phg}^m(\mathcal{M};\C^{N\times N})) ~\text{satisfies}~ \sigma_m(Q(t,s))=q(t,s,.).
\ee
From  \eqref{commutor} and pseudodifferential calculus, we now have
%
\be
    \label{Q}
    \begin{split}
        \pa_t Q(t,s)&=ic[\Lambda Id_{N\times N},Q(t,s)]+W^*_0(t)Q(t,s)+Q(t,s)W_0(t)+\widetilde{R}(t,s)\\
        &=iH(t)^*Q(t,s)-iQ(t,s)H(t)+\widetilde{R}(t,s),
    \end{split}
\ee
with ~$\widetilde{R}\in \mathcal{C}^\infty ((0,T)^2;\Psi_{phg}^{m-1}(\mathcal{M};\C^{N\times N}))$.
Next we estimate the remainder
~$R(t,s)=Q(t,s)-P(t,s)$. Set
\be
    T(t,s)=S(t,s)^*(Q(t,s)-P(t,s))S(t,s)=S(t,s)^*Q(t,s)S(t,s)-P_m(s),
\ee
so that we have, in view of \eqref{Q},
\be
    \begin{split}
        \pa_t T(t,s)&=\pa_t \left[S(t,s)^*Q(t,s)S(t,s)\right]\\
        &=S(t,s)^*\left[-i H(t)^*Q(t,s)+\pa_t Q(t,s)+iQ(t,s)H(t)\right]S(t,s)\\
        &=S(t,s)^*\widetilde{R}(t,s)S(t,s).
    \end{split}
\ee
Thus, we obtain
\be
    Q(t,s)-P(t,s)=S(s,t)^*\left[Q(s,s)-P_m(s)+\int_s^tS(\tau,s)^*\widetilde{R}(\tau,s)S(\tau,s)d\tau \right]S(s,t),
\ee
 where $\widetilde{R}\in \mathcal{C}^\infty((0,T)^2;\Psi_{phg}^{m-1}(\mathcal{M};\C^{N\times N}))$,
 $Q(s,s)-P_m(s)\in \mathcal{C}^\infty((0,T);\Psi_{phg}^{m-1}(\mathcal{M};\C^{N\times N}))$.  Therefore,  it implies
$Q(t,s)-P(t,s) \in \Bo((0,T)^2,\mathcal{L}((H^\sigma)^N,(H^{\sigma+1-m})^N)$ and
      $ \pa_t (Q(t,s)-P(t,s)),\pa_s(Q(t,s)-P(t,s)) \in \Bo((0,T)^2,\mathcal{L}((H^\sigma)^N,(H^{\sigma-m})^N)$
 for any $\sigma\in \R$.
 Together with the expression of~$Q$~ in \eqref{Qq}, we finish the proof of System Egorov Theorem.
 \end{proof}
\begin{rem}
Note that the previous Egorov Theorem implies some propagation of microlocal defect measure, as described in Burq-Lebeau \cite{BurqLebeau01} in the more complicated case of domain with boundary. The equation we obtain for $R_{1}$ is actually closely related to the trivialization of the bundle that they describe in \cite[Section 3.2]{BurqLebeau01}. In the present article, we will prove that Theorem \ref{system egorov} implies a link between the observability of the ODE and the observability of the PDE. Similarly, we plan to prove in the future \cite{CuiLaurentboundary} how the propagation of microlocal defect measure in \cite{BurqLebeau01} implies a similar link.
\end{rem}

\subsection{$N\times N$ Sharp G{\aa}rding Inequality}\lab{subsectionGarding}
In this section, we state without proof some uniform $N\times N$~type Sharp G{\aa}rding Inequality on a compact manifold. This is the equivalent of the sharp G{\aa}rding inequality in $\R^n$ as stated in \cite{Vaillancourt70}. The following versions on a compact manifold can easily be obtained frome the one on $\R^n$ by localization in local charts. We refer for instance to \cite[Section A.4]{LaurentLeautaud16} for some details in the scalar case, the argument being exactly the same.
\begin{thm}
\label{cor:unif-bound}
Assume that ~$A_t \in \mathcal{C}^0([T_1, T_2]; \Psi^0_{phg}(\mathcal{M};\C^{N\times N}))$. 
If ~ $\sigma_0(A_t)$ is nonnegative Hermitian matrix of order 0 on ~$[T_1, T_2] \times T^*\mathcal{M}$, then
\begin{align}\label{e:garding-unif1}
(A_t u , u)_{(L^2)^N} \geq - C \|u\|_{(H^{ -1/2})^N}^2 , \quad \forall t \in [T_1, T_2], ~~ u \in (H^1)^N,
\end{align}
where $C$ is independent with $u$.
\end{thm}
The proof of Theorem \ref{cor:unif-bound} is a direct consequence of the following theorem.
\begin{thm}
\label{th: sharp garding manifold}
For $A\in \Psi^0_{phg}(\mathcal{M}; \C^{N\times N})$, if $\sigma_0(A)$ is a nonnegative hermitian matrix of order 0 on $T^*\mathcal{M}$,  then there exist $C>0$ such that
\begin{align}
\label{e:garding-unif}
 (A u , u)_{(L^2)^N} \geq - C \|u\|_{(H^{-1/2})^N}^2,  \quad \forall u\in (L^2)^N.
\end{align}
\end{thm}

\section{Proof of main results}\lab{sectionproof}

In this section, we give a proof of Theorem \ref{mainthm-w}. Our proof is inspired from \cite{LaurentLeautaud16}. We start by writing the System \eqref{mainwave1.1} as a $2N\times 2N$ system of order $1$. Then, we use a trick due to Taylor to eliminate the lower order terms. Applying System Egorov Theorem, $N\times N$ G{\aa}rding inequality and control theory of ODE, we construct a connection between pseudodifferential representation and Gramian matrix of ODE System \eqref{mainmainode}.

In the following, we will be slightly more general than in Theorem \ref{mainthm-w}, in the sense that we will allow $L$ and $D$ to be pseudodifferential operators in space and not only differential operators. We assume
\be
    \lab{L}
            L=   A_0\pa_t + A_1,
\ee
with
 $A_0\in \mathcal{C}^\infty(\R;\Psi^0(\mathcal{M};\C^{N\times N}))$,
$ A_1\in \mathcal{C}^\infty(\R;\Psi^1(\mathcal{M};\C^{N\times N})).
$
and
\be
    \lab{D}
            D=   D_0\pa_t + D_1,
 \ee
 with
$D_0\in \mathcal{C}^\infty(\R;\Psi^0(\mathcal{M};\C^{K\times N}))$,
$ D_1\in \mathcal{C}^\infty(\R;\Psi^1(\mathcal{M};\C^{K\times N})).$

\subsection{Gramian operator}\lab{subsectGramian}

\noindent{\bf $\bullet $ Half wave decomposition}

We rewrite System \eqref{mainwave1.1}
to  Klein-Gordon type equations \cite{DehmanLebeau09,LaurentLeautaud16},
\be
\lab{obs-abstract}
    \begin{cases}
    (\pa_t^2 -\Delta_g)V+V+B_0\pa_tV+B_1V=0,\\
    (V(0),\partial_tV(0))=(V_0,V_1).
    \end{cases}
\ee
where
~$B_0=A_0,B_1=A_1-Id_{N\times N}$.

We set
\be\lab{halfwaveoperator}
    V_+=(\pa_t+i\Lambda)V,\quad V_-=(\pa_t-i\Lambda)V
\ee
so that
\be\lab{4.3.3}
    V_0=\f{\Lambda^{-1}}{2i}(V_+(0)-V_-(0)),V_1=\f{1}{2}(V_+(0)+V_-(0)).
\ee
 We define the map ~$\Sigma$:
\be
    \lab{half-wave}
    \begin{split}
        \Sigma: (H^s)^N\times (H^{s-1})^N &\rightarrow (H^{s-1})^{2N} ,\\
        (V_0,V_1)&\mapsto (V_+(0),V_-(0)).
    \end{split}
\ee
According to \eqref{4.3.3}, we have:
\be
    \Sigma=\left(
                     \begin{array}{cc}
                       i\Lambda Id_{N\times N} & Id_{N\times N} \\
                       -i\Lambda Id_{N\times N} & Id_{N\times N} \\
                     \end{array}
            \right)
,\quad
    \Sigma^{-1}=\f{1}{2}\left(
                     \begin{array}{cc}
                       -i\Lambda^{-1} Id_{N\times N} & i\Lambda^{-1} Id_{N\times N} \\
                        Id_{N\times N} & Id_{N\times N} \\
                     \end{array}
                    \right),
\ee
where the operator ~$\Sigma$~ is (almost) an isometry
from $ (H^s)^N\times (H^{s-1})^N $ to $ (H^{s-1})^{2N}$.

 Note that for $(V_+(0),V_-(0))=\Sigma(V_0,V_1)$, we have
 \be\lab{halfwaveeq}
 2   \|(V_0,V_1)\|^2_{(H^s)^N\times (H^{s-1})^N }=\|(V_+(0),V_-(0))\|^2_{(H^{s-1})^{2N}}.
 \ee
Let
\be\lab{B}
    B_+=\f{1}{2}(B_0-iB_1\Lambda^{-1}),\quad B_-=\f{1}{2}(B_0+iB_1\Lambda^{-1}),
\ee
we rewrite System \eqref{obs-abstract} as a $2N\times 2N$ system
\be
\lab{chapter5wave1.4}
    \begin{cases}
        (\pa_t-i\Lambda)V_++B_+V_++B_-V_-=0,\\
        (\pa_t+i\Lambda)V_-+B_+V_++B_-V_-=0,
    \end{cases}
\ee
since
$ \pa_t^2 -\Delta_g+1=(\pa_t-i\Lambda)(\pa_t+i\Lambda)$.
Denote
\be\lab{chapter five M}
\begin{split}
P=\pa_t+M_1+B,\quad
                  M_1=\left(
                   \begin{array}{cc}
                     -i\Lambda Id_{N\times N} & 0 \\
                     0 & i\Lambda Id_{N\times N} \\
                   \end{array}
                 \right),\quad B=\left(
     \begin{array}{cc}
       B_+ & B_- \\
       B_+ & B_- \\
     \end{array}
   \right).
\end{split}
\ee
Then ~$P\mathcal{V}=0, ~\mathcal{V}=(V_+,V_-)^{tr}$.
We define~$\mathfrak{S}(t,s)$ as the solution operator of System \eqref{chapter5wave1.4}.
The well-posedness of Hyperbolic System \eqref{chapter5wave1.4} yields
$\mathfrak{S}(t,s)\in \mathcal{B}((0,T)^2;\mathcal{L}(H^\sigma(\mathcal{M};\C^{2N}))),$
\ $\pa_t \mathfrak{S}(t,s),\pa_s \mathfrak{S}(t,s)\in  \mathcal{B}((0,T)^2;\mathcal{L}(H^\sigma(\mathcal{M};\C^{2N});H^{\sigma-1}(\mathcal{M};\C^{2N})))$ for all  $\sigma \in \R.$

\begin{lem}
\lab{chapter5decomlem}
Denote by ~$S_\pm(t,s)$~solution operator of
~$(\pa_t\mp i\Lambda)+B_\pm $ and
let   ~\bnan
\label{defScal}
\mathcal{S}(t,s)= \left(
                      \begin{array}{cc}
                        S_+(t,s) & 0 \\
                        0 & S_-(t,s) \\
                      \end{array}
                    \right).
										\enan
The solution operator ~$\mathfrak{S}(t,s)$~ of System \eqref{chapter5wave1.4} have the following decomposition
\be
\mathfrak{S}(t,s)=\mathcal{S}(t,s)+ \mathcal{R}(t,s),
\ee
where, for all $\sigma\in \R$,
 \be\lab{RRRR}
 \mathcal{R}(t,s)\in  \mathcal{B}((0,T)^2;\mathcal{L}(H^\sigma(\mathcal{M};\C^{2N});H^{\sigma+1}(\mathcal{M};\C^{2N})))\ee
\end{lem}
\bnp For the case of  $N=1$, we refer to \cite{LaurentLeautaud16}.

 We use a trick to decouple the equations. More precisely, we find an operator ~$K\in \mathcal{C}^\infty(0,T;\Psi^{-1}(M;\C^{2N\times 2N}))$ so that $W=(Id_{2N\times 2N}-K)\mathcal{V}$ solves a diagonal system, up to appropriate remainders. We have on the one hand
\be\nonumber
    (Id_{2N\times 2N}+K)W=\mathcal{V}-K^2\mathcal{V}.
\ee
Notice that ~$P \mathcal{V} = 0$, then
\be\nonumber\begin{split}
    (Id_{2N\times 2N}-K)P(Id_{2N\times 2N}+K)W&=(Id_{2N\times 2N}-K)P(\mathcal{V}-K^2\mathcal{V}) \\
    &=-(Id_{2N\times 2N}-K)PK^2\mathcal{V}=R\mathcal{V}.
\end{split}
\ee
Moreover, the remainder satisfies
$ R \in \mathfrak{R}^{-1}$, where
  $$
  \mathfrak{R}^{-1} =  \mathcal{C}^\infty(0,T ; \Psi_{phg}^{-1}(M;\C^{2N\times2N})) + \mathcal{C}^\infty(0,T ; \Psi_{phg}^{-2}(\mathcal{M};\C^{2N\times 2N})) \partial_t
 $$
is the admissible class of remainders in the present context.
On the other hand, we have
\be
\label{e:decoupling2}
(Id_{2N\times 2N} - K)P(Id_{2N\times2N} + K)W = PW + [P,K]W - KPK W ,
\ee
 with ~$KPK \in  \mathfrak{R}^{-1}$. We then remark that  ~$[{\pa_t} , K] W = ({\pa_t} K) W$ so that  ~$$[\pa_t , K] \in  \mathcal{C}^\infty(0,  {\color{blue}T} ; \Psi_{phg}^{-1}(\mathcal{M};\C^{2N\times2N})) \subset \mathfrak{R}^{-1}$$  and as well~$[B, K] \in \mathfrak{R}^{-1}$. Hence, if we find ~$K$~ such that
\be\displaystyle
\label{e:decouplingK}
 \left(
\begin{array}{cc}
0 & B_-\\
B_+ & 0
\end{array}
\right)  +
[M_1 , K]\in \mathfrak{R}^{-1} ,
\ee
then~ $W$~ solves the following equation
\be
\label{e:eq-W}
P_d W= R_1 W + R_2 \mathcal{V} =R\mathcal{V} ,
\ee
with ~$R_1 ,R_2 , R \in \mathfrak{R}^{-1}$ and, with $M_1$ defined in \eqref{chapter five M},
\be
P_d = \pa_t + M_1 + A_d
, \qquad A_d =
\left(
\begin{array}{cc}
B_+ & 0\\
0     & B_-
\end{array}
\right) .
\ee
Now taking
\be
K := \frac{1}{2i} \left(
\begin{array}{cc}
0 &   \Lambda^{-1}B_-\\
  -B_+ \Lambda^{-1}    & 0
\end{array}
\right)
\in
\mathcal{C}^\infty(0,T_0 ; \Psi_{phg}^{-1}(\mathcal{M};\C^{2N\times2N}))
\ee
realizes~\eqref{e:decouplingK}, and we are left to study $P_d W= R\mathcal{V}$, $R \in \mathfrak{R}^{-1}$, with ~$W=(Id_{2N\times2N} - K)\mathcal{V}$. Note that it is crucial at this step that $M_1$ is diagonal so that, for instance,   $\Lambda^{-1}B_{-}\Lambda -B_{-} \in \mathfrak{R}^{-1}$.  $\mathcal{S}(t,s)$ defined in~\eqref{defScal} is therefore the solution operator of $P_d$. Equation \eqref{e:eq-W} is now solved by
\be
W(t) = \mathcal{S}(t,s)W(s) + \int_s^t \mathcal{S}(t,t') R(t') \mathcal{V}(t') dt' , \qquad R \in \mathfrak{R}^{-1} .
\ee
 Recalling that $W=(Id_{2N\times2N} - K)\mathcal{V}$, and that~ $\mathcal{V}(t) = \mathfrak{S}(t,s)\mathcal{V}(s)$, this yields
$$
\mathcal{V}(t) = \mathcal{S}(t,s)\mathcal{V}(s) + K(t)\mathfrak{S}(t,s)\mathcal{V}(s) - \mathcal{S}(t,s)K(s)\mathcal{V}(s)+\Big( \int_s^t \mathcal{S}(t,t') R(t') \mathfrak{S}(t',s) dt' \Big) \mathcal{V}(s).
$$
This can be rewritten as
$$
\mathcal{V}(t) = \mathcal{S}(t,s)\mathcal{V}(s) + \mathcal{R}(t,s)\mathcal{V}(s),
$$
with
$$
\mathcal{R}(t,s) = K(t)\mathfrak{S}(t,s)  - \mathcal{S}(t,s)K(s) +\Big( \int_s^t \mathcal{S}(t,t') R(t') \mathfrak{S}(t',s) dt' \Big). $$
satisfying
$\mathcal{R}(t,s)  \in \mathcal{B}( (0,T_0)^2 ;\mathcal{L}(H^\sigma (\mathcal{M}; \C^{2N}); H^{\sigma+1} (\mathcal{M}; \C^{2N}) )) $ and
$\pa_t \mathcal{R}(t,s), \pa_s \mathcal{R}(t,s)  \in  \mathcal{B}( (0,T_0)^2;\mathcal{L}(H^\sigma (\mathcal{M}; \C^{2N}) ))$
for any $\sigma \in \R$, according to regularity of $\mathfrak{S}(t,s) $ and $K(s)$.
 This finishes the proof of Lemma \ref{chapter5decomlem}.
 \enp

Lemma \ref{chapter5decomlem} states that ~$\mathfrak{S}(t,s)$~can be divided into two parts, diagonal term $\mathcal{S}(t,s)$ and a more regular term ~$R(t,s)$. So we have a high-frequency representation formula for solutions of System \eqref{mainwave1.1}.\\

%

\noindent{\bf $\bullet$~  Gramian Operator}

In this part, we apply System Egorov Theorem \ref{system egorov} to express the Gramian operator as  a pseudodifferential operator, following \cite{DehmanLebeau09} for the scalar case.
\begin{thm}
\label{thmHUM}
Let  $V_*:=(V_0,V_1)\in (H^1)^N\times (L^2)^N$ be the  initial data of System \eqref{mainwave1.1}. Let $T_0>0$.  Then for any $T\in (0,T_0]$, we have
\be
    \lab{obs-2.0}
    \int_0^T\|DV(t)\|_{(L^2)^K}^2dt=((G_T+\tilde{\mathcal{R}}_T)\Sigma V_*,\Sigma V_*)_{(L^2)^N\times (L^2)^N },
\ee
 where ~$G_T\in \mathcal{C}^\infty(0,T_0;\Psi_{phg}^0(\mathcal{M};\C^{2N\times 2N}))$ and  $ \tilde{\mathcal{R}}_{T}\in \mathcal{B}(0,T_0;\mathcal{L}((H^{\sigma})^{2N},(H^{\sigma+1})^{2N}))$ is in a class of regularizing operators of order at least one.
Moreover,  the  principal symbol of $G_{T}$ can be characterized as follows:
\be\displaystyle
\lab{mainGri}
 \begin{split}
    \sigma_0(G_T)&=\left(
                    \begin{array}{cc}
                  G_{\rho}^+(T)  & 0 \\
                      0 & G_{\rho}^-(T)   \\
                    \end{array}
                  \right)\in \mathcal{C}^\infty(0,T_0;S_{phg}^0(\mathcal{M};\C^{2N\times 2N})),\\
     G_{\rho}^+(T)&=\f{1}{4}\int_0^TR^{*}_+(0,t;\varphi_{-t}(\rho))d_+^{*}(t,\varphi_{-t}(\rho))d_+(t,\varphi_{-t}(\rho))R_+(0,t;\varphi_{-t}(\rho)) dt, \\
     G_{\rho}^-(T)&= \f{1}{4}\int_0^TR^{*}_-(0,t;\varphi_{t}(\rho))d_-^{*}(t,\varphi_{t}(\rho))d_-(t,\varphi_{t}(\rho))R_-(0,t;\varphi_{t}(\rho))dt,
 \end{split}
 \ee
 where $R_\pm(\tau,t;\rho)$ satisfies
\be
\lab{main-r}
\begin{split}
\f{d R_\pm(\tau,t;\rho)}{d \tau}=R_\pm(\tau,t;\rho)b_\pm(\tau,\varphi_{\pm(t-\tau)}(\rho)), \quad
R_\pm(t,t;\rho)=Id_{N\times N},
\end{split}
\ee
with $b_\pm=\sigma_0(B_\pm)=\f{1}{2} (a_0\pm \f{a_1}{i|\xi|_x})$, $d_\pm= d_0\pm\f{d_1}{i|\xi|_x}$ and $\varphi_t(\rho)$ is the Hamiltonnian flow of $|\xi|_{x}$ initiated at $\rho$ (see Theorem \ref{system egorov} for more precisions).
\end{thm}
This theorem is the main step to prove Theorem \ref{thmHUMintro}. The only difference is the characterization of $G_{\rho}^{\pm}(T)$ as the Gramian matrix of appropriate control problems, which will be made in another Section. The proof is a direct combination of Proposition \ref{prop1} and \ref{propHUM2} below.
\begin{prop}
\lab{prop1}
Denote by $V_*=(V_0,V_1)\in (H^1)^N\times (L^2)^N$ the initial data of System \eqref{mainwave1.1}. We have
\be
    \lab{obsprop}
    \int_0^T\|DV(t)\|_{(L^2)^K}^2dt=((\mathcal{G}_T+\mathcal{R}_T)\Sigma V_*,\Sigma V_*)_{(L^2)^{2N} },
\ee
 where
\be\lab{prop-mathR_T}
    \mathcal{R}_T\in \mathcal{B}_{loc}(\R^+;\mathcal{L}(H^\sigma(\mathcal{M};\C^{2N});H^{\sigma+1}(\mathcal{M};\C^{2N}))), \quad \forall \sigma\in \R.
\ee
and
 \be
 \lab{G_T}\begin{split}
  &  \mathcal{G}_T=\int_0^T\left(
                          \begin{array}{cc}
                            S(t,0)^*_+D^{11}S(t,0)_+ & S(t,0)^*_+D^{12}S(t,0)_- \\
                            S(t,0)^*_-D^{21}S(t,0)_+ & S(t,0)^*_-D^{22}S(t,0)_- \\
                          \end{array}
                        \right)dt,\\
&D^{11}= \f{D_0^*D_0}{4}+\f{\Lambda^{-1}D_1^*D_1\Lambda^{-1}}{4}-\f{\Lambda^{-1}D_1^*D_0}{4i}+\f{D_0^*D_1\Lambda^{-1}}{4i},\\
&D^{12}= \f{D_0^*D_0}{4}-\f{\Lambda^{-1}D_1^*D_1\Lambda^{-1}}{4}-\f{\Lambda^{-1}D_1^*D_0}{4i}-\f{D_0^*D_1\Lambda^{-1}}{4i},\\
&D^{21}= \f{D_0^*D_0}{4}-\f{\Lambda^{-1}D_1^*D_1\Lambda^{-1}}{4}+\f{\Lambda^{-1}D_1^*D_0}{4i}+\f{D_0^*D_1\Lambda^{-1}}{4i},\\
&D^{22}= \f{D_0^*D_0}{4}+\f{\Lambda^{-1}D_1^*D_1\Lambda^{-1}}{4}+\f{\Lambda^{-1}D_1^*D_0}{4i}-\f{D_0^*D_1\Lambda^{-1}}{4i},
\end{split}\ee
where the definition of $S_{\pm}(t,s)$ is given in Lemma \ref{chapter5decomlem},
\end{prop}
\begin{proof}
The proof of Proposition \ref{prop1} essentially relies on some computations and an application of Lemma \ref{chapter5decomlem}.  According to \eqref{4.3.3}
\be
\begin{split}
\int_0^T (DV,DV)_{(L^2)^K}dt=&\int_0^T
 \left(\f{D_1\Lambda^{-1}}{2i}(V_+-V_-),\f{D_1\Lambda^{-1}}{2i}(V_+-V_-)\right)_{( L^2)^K} \\
 &+\left(\f{D_0}{2}(V_++V_-),\f{D_0}{2}(V_++V_-)\right)_{( L^2)^K}\\
 &+\left(\f{D_0}{2}(V_++V_-),\f{D_1\Lambda^{-1}}{2i}(V_+-V_-)\right)_{( L^2)^K}\\
&+\left(\f{D_1\Lambda^{-1}}{2i}(V_+-V_-),\f{D_0}{2}(V_++V_-)\right)_{( L^2)^K}dt.\\
\end{split}
\ee
Denote ~$\hat{D}^1=\left(
                          \begin{array}{cc}
                             D^{11}&  D^{12} \\
                             D^{21}&  D^{22}\\
                          \end{array}
                        \right)$.
  Since ~$V_*=(V_0,V_1)'$, we have
\be\lab{prop-obs-equ}
\int_0^T (DV,DV)_{(L^2)^K}dt=\int_0^T\left(\hat{D}^1\mathfrak{S}(t,0)\Sigma V_*,\mathfrak{S}(t,0)\Sigma V_*\right)_{(L^2)^{2N} }.
\ee
  According to Lemma \ref{chapter5decomlem},
\be\lab{prop-obs-decom}
(V_+,V_-)^{tr}=\mathfrak{S}(t,0)\Sigma V_*=(\mathcal{S}(t,0)+R(t,0))\Sigma V_*.
                      \ee
     Combining  \eqref{prop-obs-decom} with \eqref{prop-obs-equ}, we have
       \beq\begin{split}
\int_0^T (DV,DV)_{(L^2)^K} dt=&\int_0^T (\mathcal{S}^*(t,0) \hat{D}^1 \mathcal{S}(t,0) \Sigma V_*,\Sigma V_*)_{(L^2)^{2N}}\\
                                                          &+(R^*(t,0)  \hat{D}^1 \mathcal{S}(t,0)+\mathcal{S}^*(t,0) \hat{D}^1R(t,0)\Sigma V_*,\Sigma V_*)_{(L^2)^{2N}}\\
                                                        &+(R^*(t,0) \hat{D}^1R(t,0)\Sigma V_*,\Sigma V_*)_{(L^2)^{2N} }.
\end{split}
\eeq
Define
\be
\nonumber
\mathcal{R}_T=R^*(t,0)  \hat{D}^1 \mathcal{S}(t,0)+\mathcal{S}^*(t,0) \hat{D}^1R(t,0)+R^*(t,0) \hat{D}^1R(t,0)
\ee
with
\be
\nonumber
    \begin{split}
        \mathcal{G}_T     =\mathcal{S}^*(t,0) \hat{D}^1 \mathcal{S}(t,0),
    \end{split}
\ee
we obtain \eqref{obsprop}.
We claim that $\mathcal{R}_T$ satisfies \eqref{prop-mathR_T}.
Indeed,
 $\mathcal{S}(t,0)$ preserves the
regularity thanks to \eqref{RRRR} in Lemma \ref{chapter5decomlem} and $\hat{D}^1$ is a Pseudodifferential operator of order 0.
\enp
\begin{prop}
\label{propHUM2}
 $\mathcal{G}_T$ (defined in \eqref{G_T}) has a decomposition as  ~$\mathcal{G}_T=G_T+R_T$ , ~where $R_T$~ satisfies
$R_T\in \mathcal{B}_{loc}(\R^+;\mathcal{L}(H^\sigma(\mathcal{M};\C^{2N});H^{\sigma+1}(\mathcal{M};\C^{2N})))$ for all $\sigma\in \R$
and ~$G_T\in \mathcal{C}^\infty(\R^+;\Psi_{phg}^0(\mathcal{M};\C^{2N\times 2N}))$ has principal symbol
{
\be\displaystyle
\lab{mainGriprop}
 \begin{split}
    \sigma_0(G_T)&=\left(
                    \begin{array}{cc}
                  G_{\rho}^+(T)  & 0 \\
                      0 & G_{\rho}^-(T)   \\
                    \end{array}
                  \right), \\
     G_{\rho}^{\pm}(T)&=\f{1}{4}\int_0^TR^{*}_{\pm}(0,t;\varphi^{\mp}_{t}(\rho))d_{\pm}^{*}(t,\varphi^{\mp}_{t}(\rho))d_{\pm}(t,\varphi^{\mp}_{t}(\rho))R_{\pm}(0,t;\varphi^{\mp}_{t}(\rho)) dt, \\
 \end{split}
 \ee
 where $R_\pm(s,t;\rho)$ satisfies
\be
\lab{main-rprop}
\begin{split}
\f{dR_\pm(\tau,t;\rho)}{d\tau}=R_\pm(\tau,t;\rho)b_\pm(\tau,\varphi^{\mp}_{\tau-t}(\rho)),\quad
R_\pm(t,t;\rho)=Id_{N\times N},
\end{split}
\ee
}
with $b_\pm=\sigma_0(B_\pm)=\f{1}{2} a_0\pm\f{1}{2}\f{a_1}{i|\xi|_x}$, $d_\pm= d_0\pm\f{d_1}{i|\xi|_x}$ and $\varphi^{\mp}_{t}(\rho)$ is the Hamiltonian flow of ~$\mp |\xi|_x $ initiated at $\rho$.
\end{prop}

The proof relies on Egorov Theorem and the following Lemma that deals with the anti-diagonal terms of the Gramian control operator whose proof is postponed to the Appendix.
\begin{lem}
\lab{lemma: Regularity}
Assume that $\mathcal{I} $ is an interval in ~$\R$, let~$$H_\pm(t)=\pm\Lambda Id_{N\times N}+iW_0(t),$$
 with $W_0\in \mathcal{C}^\infty (0,T;\Psi^0_{phg}(\mathcal{M};\C^{N\times N})),$
then for any  $B_0 \in \mathcal{C}^\infty (0,T;\Psi^m_{phg}(\mathcal{M};\C^{N\times N}))$, $m \in \R$, we can define
\begin{align*}
    B(T) = \int_0^T  S_\pm(t,0)^* B_0 S_\mp(t,0) d t ,
\end{align*}
and we have $B \in \mathcal{B}_{loc}(0,T;\mathcal{L}((H^{\sigma})^N, (H^{\sigma+1-m })^N)$ for all $\sigma  \in \R$.
\end{lem}

\begin{proof}[Proof of Proposition \ref{propHUM2}]
Integrating  the anti-diagonal terms of $\mathcal{G}_T$ in \eqref{G_T}  on $[0,T]$ yields
\be\begin{split}
\int_0^TS(t,0)^*_+ D^{12}S(t,0)_- dt \in \mathcal{B}_{loc}(\mathcal{I} ;\mathcal{L}((H^{\sigma})^N, (H^{\sigma+1-m })^N),\\
\int_0^TS(t,0)^*_- D^{21}S(t,0)_+ dt \in \mathcal{B}_{loc}(\mathcal{I} ;\mathcal{L}((H^{\sigma})^N, (H^{\sigma+1-m })^N).
\end{split}
\ee
Next we claim that there exist $G_{\rho}^+(T),G_{\rho}^-(T)$ satisfying \eqref{mainGriprop} and
{\be\lab{xxxxxx} \begin{split}
\int_0^TS(t,0)^*_+D^{11}S(t,0)_+dt-G_{\rho}^+(T)\in \mathcal{B}_{loc}(\mathcal{I} ;\mathcal{L}((H^{\sigma})^N, (H^{\sigma+1-m })^N),\\
\int_0^T S(t,0)^*_-D^{22}S(t,0)_-dt-G_{\rho}^-(T)\in \mathcal{B}_{loc}(\mathcal{I} ;\mathcal{L}((H^{\sigma})^N, (H^{\sigma+1-m })^N).
\end{split}
\ee}
We only detail $S(t,0)^*_+D^{11}S(t,0)_+$, the other case being similar.
Using Theorem \ref{system egorov} with
\begin{itemize}
\item $c=1$, so that $\chi_{\tau,s}^{1}(\rho)=\varphi_{\tau-s}^{-}(\rho)=\varphi_{s-\tau}(\rho)$ and $\chi_{t,0}^{c=1}=\varphi_{-t}$
\item $W_{0}=B_{+}$
 \item { $P_{m}=D^{11}=\f{1}{2}(D_0+\f{D_1\Lambda^{-1}}{i})^*\cdot\f{1}{2}(D_0+\f{D_1\Lambda^{-1}}{i})$}
\end{itemize}
gives $S(t,0)^*_+D^{11}S(t,0)_+= R^{*}_+(0,t,\varphi_{-t}(\rho))d^{11}(t,\varphi_{-t}(\rho))R_+(0,t,\varphi_{-t}(\rho))$, where $R_{+}$ solves
\be
 \f{d R_+(\tau,s;\rho)}{d\tau}=R_+(\tau,s;\rho)b_{+}(\tau,\varphi^{-}_{\tau-s}(\rho)), \quad  R_+(s,s;\rho)=Id_{N\times N}
       \ee
and $b_{+}=\sigma_0(B_+)=\f{1}{2}\sigma_0(A_0-iA_1\Lambda^{-1})=\f{1}{2}(a_0+\f{a_1}{i|\xi|_x}), d_+=\sigma_0(D_0+\f{D^1\Lambda^{-1}}{i})=d_0+\f{d_1}{i|\xi|_x}$,

 The other case $S(t,0)^*_-D^{22}S(t,0)_-$ is the same with $c=-1$, $\chi_{t,0}^{-1}=\varphi_{t}$, $W_{0}=B_{-}$, $P_{m}=D^{22}$.
 \end{proof}

\noindent{\bf $\bullet$~ Gramian operator and weak observability inequality \eqref{mainobsinw}}

As a direct consequence (or verification) of Theorem \ref{thmHUM}, $\sigma_0(G_T)$ is a nonnegative symmetric matrix. Thanks to $N\times N$ Sharp G{\aa}rding Inequality \eqref{e:garding-unif1}, we can construct a connection between weak observability inequality \eqref{mainobsinw} and Gramian control operator as follows.

\begin{prop}
\lab{main-prop-two}
Let ~$T>0$~. Define
\be
\lab{RR}\begin{split}
\mathcal{R}_2(T)=\min\Big\{\min\limits_{\rho_0\in S^*\mathcal{M}}\sup\left\{s\in \R\left|\beta^{*}(G_{\rho}^{+}(T)-sId_{N\times N})\beta\geq 0, \forall \beta\in \C^N\right.\right\},  \\  \min\limits_{\rho_0\in S^*\mathcal{M}}\sup\left\{s\in \R\left|\beta^{*}(G_{\rho}^{-}(T)-sId_{N\times N})\beta\geq 0, \forall \beta\in \C^N\right.\right\} \Big\}
\end{split}
\ee
and $G_{\rho}^\pm(T)$ are defined in \eqref{mainGri}. If for all $\rho_0\in S^*\mathcal{M}$, ~$\sigma_0(G_T)$~is a symmetric positive matrix, then the weak observability inequality \eqref{mainobsinw} holds for all solutions of System \eqref{mainwave1.1} with
\be
C^2_{obs}(T)\geq \f{1}{2\mathcal{R}_2(T)}.
\ee
\end{prop}
\bnp
According to Theorem \ref{thmHUM}, we have
\be\lab{prostep1}
\begin{split}
    \int_0^T\|DV(t)\|_{(L^2)^K}^2dt&=\left((G_T+\mathcal{R}_T)\Sigma V_*,\Sigma V_*\right)_{(L^2)^{2N} }\\
    &\geq (G_T\Sigma V_*,\Sigma V_*)_{(L^2)^{2N} }-C^1\|\Sigma V_*\|_{ (H^{-\f{1}{2}})^{2N} }^2,
\end{split}
\ee
 and $\sigma(G_T)$ is a symmetric positive matrix.
Then $\sigma(G_T)-\mathcal{R}_2(T)Id_{2N\times 2N}$ is a nonnegative symmetric matrix (here, we are using that actually, the supremum in \eqref{RR} is actually a maximum). By
$N\times N$ Sharp G{\aa}rding Inequality \eqref{e:garding-unif1} and
\eqref{prostep1},
we obtain
\beq\displaystyle
\int_0^T\|DV(t)\|_{(L^2)^K}^2dt\geq (\mathcal{R}_2(T)\Sigma V_*,\Sigma V_*)_{(L^2)^{2N} }-C^1 \|\Sigma V_*\|^2_{(H^{-\f{1}{2}})^{2N}}.
\eeq
Combining with \eqref{halfwaveeq}, we have
\be\displaystyle
\lab{main-prop-two-last-ineq}
\begin{split}
\frac{1}{\mathcal{R}_2(T)}\int_0^T\|DV(t)\|_{(L^2)^K}^2dt\geq 2\|(V_0,V_1)\|^2_{(H^1)^N\times (L^2)^N}-C\|(V_0,V_1)\|^2_{(H^{\f{1}{2}})^N\times (H^{-\f{1}{2}})^N}.
\end{split}
\ee
 So this finishes the proof of Proposition \ref{main-prop-two}.
 \enp
\noindent{\bf  $\bullet$~Gramian control operator and controllability of ODE system}

For fixed $\rho_{0}\in S^*\mathcal{M}$, we will consider the following control system
\be
    \lab{chapter5mainodepm}
\left\{ \begin{array}{lll}
        \displaystyle \dot{ X}(t)=\f{1}{2}a^{*}_{\pm}(t,\varphi_{ \mp t}(\rho_0))X(t)+\f{1}{2}d^{*}_{\pm}(t,\varphi_{ \mp t}(\rho_0))u(t),  \\
        X(0)=X_0\in \C^N,
    \end{array}\right.
\ee
where $X(t)=(X_1,\cdots,X_N)^{tr}$ is a vector having $N$ components, and ~$a_{\pm}=a_0{\pm}\f{a_1}{i|\xi|_x}$ is a $N\times N$ matrix. $d_{\pm}=d_0{\pm}\f{d_1}{i|\xi|_x}$ is a $K\times N$ matrix. $u(t)\in L^2(0,T;\C^{K})$ is the control.

Next we reveal connections between Gramian control operator and exact controllability of ODE System \eqref{chapter5mainodepm} as follows.

The first step is an elementary but crucial Lemma.
\begin{lem} \label{lmresolr}Let $\rho_0\in S^*\mathcal{M}$.
Denote $\tilde{R}_{\pm}(\cdot,\cdot;\rho_{0})$ the resolvent of System \eqref{chapter5mainodepm} (see \cite[Proposition 1.5]{Coron07} for definition).  Then
\be
\label{resR}
R_{\pm}(\tau,t;\rho_{0})=\widetilde{R}_{\pm}(\tau,t;\varphi_{\pm t}(\rho_0))^*,
\ee
 where $R_{\pm}(\cdot,\cdot;\rho_0)$ is defined in \eqref{main-r}.
Moreover, let the $G_{\pm}$ be the Gramian of the control System \eqref{chapter5mainodepm}(see \cite[Definition 1.10]{Coron07} ). Then
\be
\label{egalG}
 G_{\pm}= G_{\rho_0}^{\pm}(T)
 \ee where $G_{\rho_0}^{\pm}(T)$ is defined in \eqref{mainGri}.
\end{lem}
\bnp
The equation \eqref{main-r} verified by $R_{\pm}$ is
\be
\lab{main-rbis}
\begin{split}
\f{d R_\pm(\tau,t;\rho_0)}{d \tau}=R_\pm(\tau,t;\rho_0)b_\pm(\tau,\varphi_{\pm(t-\tau)}(\rho_0)), \quad
R_\pm(t,t;\rho_0)=Id_{N\times N},
\end{split}
\ee
Taking the adjoint of {\color{blue}the definition of the resolvent, see \cite [(1.10) in Proposition 1.5]{Coron07},} applied to System \eqref{chapter5mainodepm} and recalling $b_{\pm}(\tau,\rho_0)=\frac{a_{\pm}(\tau,\rho_0)}{2}$ gives
\be
\small
\nonumber
\begin{split}
\f{d\widetilde{R}_{\pm}(\tau,t;\rho_0)^*}{d\tau}=\widetilde{R}_{\pm}(\tau,t;\rho_0)^*\f{a_{\pm}(\tau,\varphi_{ \mp \tau}(\rho_0))}{2}=\widetilde{R}_{\pm}(\tau,t;\rho_0)^*b_{\pm}(\tau,\varphi_{\mp  \tau}(\rho_0)),\quad
\widetilde{R}_{\pm}(t,t;\rho_0)=Id_{N\times N}.
\end{split}
\ee
Applying at the point $\varphi_{\pm t}(\rho_0)$ gives
\be
\small
\lab{R-bis}
\begin{split}
\f{d\widetilde{R}_{\pm}(\tau,t;\varphi_{\pm t}(\rho_0))^*}{d\tau}=\widetilde{R}_{\pm}(\tau,t;\varphi_{\pm t}(\rho_0))^*b_{\pm}(\tau,\varphi_{\pm(t-  \tau)}(\rho_0)),\quad
\widetilde{R}_{\pm}(t,t;\varphi_{\pm t}(\rho_0))=Id_{N\times N}.
\end{split}
\ee
Therefore, for fixed $t$, the two matrices $R_\pm(\tau,t;\rho_0)$ and $\widetilde{R}_{\pm}(\tau,t;\varphi_{\pm t}(\rho_0))^*$ depending on $\tau$ solve the same equation with same initial data, so they are equal by the Cauchy-Lipschitz Theorem.

Concerning the second part of the Lemma, we have
\be
\nonumber
~G_{\pm}=\f{1}{4}\int_0^T\tilde{R}_{\pm}(0,s;\rho_{0})d^{*}_{\pm}(s,\varphi_{ \mp s}(\rho_0))d_{\pm}(s,\varphi_{ \mp s}(\rho_0))\tilde{R}_{\pm}^{*}(0,s;\rho_{0})ds. \ee
Concerning \eqref{mainGri}, it can be written
\be  \nonumber G_{\rho_0}^{\pm}(T)= \f{1}{4}\int_0^TR^{*}_{\pm}(0,t;\varphi_{\mp t}(\rho_0))d_{\pm}^{*}(t,\varphi_{\mp t}(\rho_0))d_{\pm}(t,\varphi_{\mp t}(\rho_0))R_{\pm}(0,t;\varphi_{\mp t}(\rho_0))dt.
\ee
Now, using the obtained identity \eqref{resR}, we get the expected result
\be  \nonumber G_{\rho_0}^{\pm}(T)= \f{1}{4}\int_0^T\widetilde{R}_{\pm}(0,t;\rho_0)d_{\pm}^{*}(t,\varphi_{\mp t}(\rho_0))d_{\pm}(t,\varphi_{\mp t}(\rho_0))\widetilde{R}_{\pm}^{*}(0,t;\rho_0)dt.
\ee
\enp
Theorem \ref{thmHUMintro} is now a direct consequence of Theorem \ref{thmHUM} and of the previous Lemma. Another consequence is the following.
 \begin{prop}\lab{propodexxxxnonsym}
Let $T>0$, for any ~ $\rho_0\in S^*\mathcal{M}$~, we have the equivalence
\begin{enumerate}
\item \label{HermnonSm}Hermitian matrix~$\sigma_0(G_{T})$ (defined in \eqref{mainGri}) is positive.
\item \label{impliquepm} System \eqref{chapter5mainodepm} are exactly controllable in both cases $+$ and $-$.
 \end{enumerate}
\end{prop}
\bnp
Thanks to \eqref{egalG}, the result is now a direct consequence of the classical equivalence between invertibility of the Gramian and controllability, {\color{blue}see \cite[Theorem 1.11]{Coron07}.}
\enp
In the next Proposition, we prove that if $a$ and $d$ have some symmetry properties, we need to check the controllability of only one system $+$ or $-$. This property will be satisfied in the two important cases
\begin{itemize}
\item $A$ is a differential operator so that $a_0$ is even and $a_1$ is odd and $a_{\pm}= \frac{1}{2}\left(a_0\pm\f{a_1}{i|\xi|_x}\right)$ (same for $d$)
\item $A=\Lambda A(x)$ where $A(x)$ is the operator of multiplication by a matrix $A(x)$, which will be the case for zero order coupling.
\end{itemize}
 \begin{prop}\lab{propodexxxx}
Assume $a_{\pm}=a_{\mp}\circ \sigma, b_{\pm}=b_{\mp}\circ \sigma,d_{\pm}=d_{\mp}\circ \sigma $,  {\color{blue}
where ~$\sigma : T^*\mathcal{M} \rightarrow T^*\mathcal{M}$ is the involution~ $(x, \xi) \mapsto (x, -\xi)$.
}
Let $T>0$, for any  $\rho_0\in S^*\mathcal{M}$~, we have the equivalence
\begin{enumerate}
\item \label{Herm}Hermitian matrix~$\sigma_0(G_{T})$ (defined in \eqref{mainGri}) is positive at $\rho_0$ and $\sigma(\rho_0)$
\item \label{implique-} System \eqref{mainmainode} is exactly controllable in the case $a_-$ and $d_-$ for $\rho_0$ and $\sigma(\rho_0)$
\item \label{implique+}System \eqref{mainmainode} is exactly controllable in the case $a_+$ and $d_+$ for $\rho_0$ and $\sigma(\rho_0)$.
\end{enumerate}
\end{prop}
The proof is direct with Proposition \ref{propodexxxxnonsym} and the following Lemma at hand.
\begin{lem}
\label{lmegalG}
With the symmetry assumptions of Proposition \ref{propodexxxx}, we have
\bna
G_{\sigma(\rho_0)}^{+}(T)=G_{\rho_0}^{-}(T).
\ena
\end{lem}
\bnp
By using \cite[Lemma B.2]{DehmanLeLeautaud14}, we have
\be\lab{equivalent-sigma}
(\varphi_{t}(\rho_0))\circ \sigma =\sigma \circ (\varphi_{-t}(\rho_0)), \quad
(\varphi_{-t}(\rho_0))\circ \sigma =\sigma \circ (\varphi_{t}(\rho_0)).
\ee
We have, first by \eqref{equivalent-sigma}, then by the symmetry property of $d$,
\be
d_{+}(t,\varphi_{- t}(\sigma(\rho_0)))=d_{+}(t,\sigma\circ \varphi_{ t}(\rho_0))=d_{-}(t,\varphi_{ t}(\rho_0))
\ee
and the same holds for the transpose. For the same reasons, we have
\be
b_+(\tau,\varphi^{-}_{\tau-t}(\sigma(\rho_0)))=b_+(\tau,\sigma\circ \varphi^{+}_{\tau-t}(\rho_0))=b_-(\tau,\varphi^{+}_{\tau-t}(\rho_0)),
\ee
In particular, we have by \eqref{main-r}
\be
\begin{split}
\f{dR_+(\tau,t;\sigma(\rho_0))}{d\tau}=R_+(\tau,t;\sigma(\rho_0))b_+(\tau,\varphi^{-}_{\tau-t}(\sigma(\rho_0)))=R_+(\tau,t;\sigma(\rho_0))b_-(\tau,\varphi^{+}_{\tau-t}(\rho_0))
\end{split}
\ee
which is the equation satisfied by $R_-(\tau,t;\rho_0)$ with same initial data, so that $R_+(\tau,t;\sigma(\rho_0))=R_-(\tau,t;\rho_0)$. \eqref{equivalent-sigma} and this symmetry of $R$ give
\be
R_+(0,t;\varphi^{-}_{t}(\sigma(\rho_0)))=R_+(0,t;\sigma\circ \varphi^{+}_{t}(\rho_0))=R_-(0,t; \varphi^{+}_{t}(\rho_0)).
\ee
This gives exactly the expected result $G_{\sigma(\rho_0)}^{+}(T)=G_{\rho_0}^{-}(T)$.
 \enp
\subsection{Proof of Theorem \ref{mainthm-w}}\lab{subsectionproofThm}
We actually plan to prove the slightly more general theorem.
Let
\be
a_0\in C^\infty(\R;S_{phg}^0(T^*\mathcal{M};\C^{N\times N})),a_1\in C^\infty(\R;S_{phg}^1(T^*\mathcal{M};\C^{N\times N})),
\ee
are, respectively, the principal symbols of $A_0$ and $A_1$ which are defined in \eqref{L}.
Let
\be
d_0\in C^\infty(\R;S_{phg}^0(T^*\mathcal{M};\C^{K\times N})),d_1\in C^\infty(\R;S_{phg}^1(T^*\mathcal{M};\C^{K\times N}))
\ee
are, respectively, the principal symbols of $D_0$ and $D_1$ which are defined in \eqref{D}.
\begin{thm}
\lab{mainthmPseudo}
Solutions of System \eqref{mainwave1.1} satisfy weak observability inequality \eqref{mainobsinw} on $[0,T]$ if and only if for any $\rho_0\in S^*\mathcal{M}$, any initial data $X_0\in \C^N$, both systems \eqref{chapter5mainodepm} are exactly controllable on $[0,T]$.
 \end{thm}

\begin{proof}
Step 1.
By Proposition \ref{main-prop-two} and Proposition \ref{propodexxxxnonsym}, it is easy to show that
\noindent{Systems \eqref{chapter5mainodepm} are exactly controllable implies weak observability inequality \eqref{mainobsinw} for all solution of System \eqref{mainwave1.1}.}\\
Step 2. We check that
\noindent{ System \eqref{mainwave1.1} satisfy weak observability inequality \eqref{mainobsinw} implies System \eqref{chapter5mainodepm} are exactly controllable.  }
Suppose by contradiction that this part of the Theorem failed, then
 there exists a $\rho_0\in S^*\mathcal{M}$~ and Hamiltonian flow~$\varphi_t(\rho_0)$~ such that one of the System \eqref{chapter5mainodepm} is not controllable, let us say $-$ for fixing the ideas.
Hence ~$G_{\rho_0}(T)$~ is  nonpositive. According to Proposition \ref{propodexxxxnonsym}, we have
\be
\det\left(G_{\rho}^-(T)\right)=0,
\ee
Then there exists a vector ~$P\in \C^N, |P|_{l^2(\C^N)}=1$~such that  ~$P^{*}G_{\rho}^-(T)P=0$~.
 We take a local chart ~$x_0\in(U_\kappa, k)$ of~$M$ so that $g_{i,j}(x_{0)}=Id$. We denote by ~$(y_0 , \eta_0)$ the coordinates of ~$\rho_0$~ in this chart.
We choose  $\psi \in \mathcal{C}^\infty_c(\R^n)$ such that $supp(\psi) \subset
\kappa(U_\kappa)$, and $\psi = 1$ in a neighborhood of~ $y_0$~. Next we define
$$
w^k (y) = C_0 k^{\frac{n}{4}} e^{i k \varphi(y)}\psi(y), \quad
\varphi(y) = y \cdot \eta_0 + i(y-y_0)^2,~~  C_0 >0 .
$$
Setting now
\begin{align}
\label{eq: sequence concentrating}
v_-^k =\kappa^* w^k \in \mathcal{C}^\infty_c( \mathcal{M}),
\end{align}
We have $v_-^k \rightharpoonup 0$ and $\lim_{k \rightarrow
  \infty} \|v_-^k\|_{L^2} =\lim_{k \rightarrow
  \infty} \|Pv_-^k\|_{(L^2)^{N}}= 1$
for an appropriate choice of~$C_0$,
while $\lim_{k \rightarrow
  \infty} \|v_-^k\|_{H^{-1}} = 0$. Moreover, a classical computation on~$(w^k)_{k \in \N}$ show that
  for all  $ A \in \Psi_{phg}^0(\mathcal{M}; \C^{N\times N})$, ~$(v_-^k)_{k \in \N}$ satisfies
\begin{equation}
\label{e:Amdm}
\left(A Pv_-^k ,  Pv_-^k \right)_{(L^{2})^N} \rightarrow \sigma_{0}(P^{*}AP)(\rho_0), ~k\rightarrow \infty  .
\end{equation}
Next, we set ~$v_+^k = 0, ~~k \in \N$, and $V^k = \Sigma^{-1} ( 0 ,  Pv_-^k)  \in (H^1)^N \times( L^{2})^N$.
Denoting $\mathcal{V}^k(t)$ the solution to System \eqref{mainwave1.1} with initial data  $V^k$, Theorem \ref{thmHUM} and~\eqref{e:Amdm} gives
\begin{equation}\lab{critical}
\begin{split}
\lim_{k \rightarrow \infty}
\int_0^T \|D \mathcal{V}^k (t)\|_{(L^2)^K}^2 dt
&= \lim_{k \rightarrow \infty}
 \big((G_T + \widetilde{R}_T) \Sigma V^k, \Sigma V^k \big)_{(L^2)^{2N}} \\
 &= \lim_{k \rightarrow \infty}
 \big( G_T  \Sigma V^k, \Sigma V^k \big)_{(L^2)^{2N}} \\
& = P^{*}G_{\rho}^-(T) P = 0 ,
\end{split}
\end{equation}
where we used that ~$R_T$ is 1-smoothing, that ~ $G_T \in \Psi_{phg}^0(\mathcal{M};\C^{2N\times 2N})$ has principal symbol given by ~\eqref{xxxxxx}, and the choice of ~$\rho=\rho_0$ in ~\eqref{xxxxxx} .
Then we obtain a contradiction and finish the proof of Theorem \ref{mainthmPseudo}.
\enp
\bnp[Proof of Theorem \ref{mainthm-w}]
Since $A$ and $D$ are differential operators, Proposition \ref{propodexxxx} applies and the conclusion is direct from Theorem \ref{mainthmPseudo}.
\enp
\subsection{Proof of Theorem \ref{theoremoptimalcons}}
\label{subsectproofthmbis}
\bnp[Proof of Theorem \ref{theoremoptimalcons}]
By Theorem \ref{mainthm-w},
 $G_{\rho_0}(T)$ is positive for any $\rho_0\in S^*\mathcal{M}$.
 Hence ~$T>T_{crit}=\inf\limits_{T_0}\{T_0|\min\limits_{\rho_0\in S^*\mathcal{M},|\beta|=1}\beta^{*} G_{\rho_0}(T_0)\beta>0\}$~.

Using Proposition \ref{propodexxxx}, we have
\be
\mathfrak{K}(T)=\mathcal{R}_2(T).
\ee

Proposition \ref{main-prop-two} then gives that the weak observability holds with \be
C^2_{obs}(T)\geq \f{1}{2\mathfrak{K}(T)}.
\ee
Since $\mathcal{M}$ is a compact manifold, it suffices to show that there exists a $\rho_0\in S^*\mathcal{M}$, such that
$\sigma(G_T)-\mathcal{R}_2(T)Id_{2N\times 2N}$ is a nonpositive matrix.
In view of the proof of Theorem \ref{mainthmPseudo}, it is easy to obtain
\be
 C^2_{obs}(T) \mathcal{R}_2(T)  =  C^2_{obs}(T)  \int_0^T \|D \mathcal{V}^k(t) \|_{(L^2)^K}^2 dt \geq \sum_k\|V_k\|^2_{L^2} \rightarrow 1.
\ee

So we finish the proof of Theorem \ref{theoremoptimalcons} thanks to  Corollary \ref{maincorfc} is a direct consequence of  \cite[Theorem 3.2]{XQX2005} and Theorem \ref{theoremoptimalcons}.
\enp
\section{Coupling of order zero}
\label{sectzero}
In this Section, we plan to prove that the involved systems are well-posed and prove Theorem \ref{thmcontrolmulti}. In a first Section, we will also describe that the assumption of the matrix being in a subdiagonal form is actually quite general, up to some change of unknown.
\subsection{Getting the subdiagonal form}
\label{sectiongettindiag}
\subsubsection{Getting the subdiagonal form in the constant coupling case}
We use the Brunovsky normal form as described in Proposition \ref{propBrunovsky}. This gives the immediate Lemma.

\begin{lem}
\label{lmequivBrunov}
Assume that $(A,B)$ are constant matrices and satisfy the Kalman rank condition. Let $(\widetilde{A},\widetilde{B})$, $\Mx$, $F$, $M_{u}$ given by Proposition \ref{propBrunovsky}. Define also the space varying matrix $\widetilde{A}_{\omega}(x)=\widetilde{A}_{\chi_{\omega}(x)}= \Mx^{-1}(A\Mx+\chi_{\omega}(x)BF)$ where $\chi_{\omega}=1$ on $\omega$. Then, if $\widetilde{U}$ is solution of
\be
   \label{systemBrunov}
        \begin{cases}
            \pa_t^2 \widetilde{U}-\Delta_g \widetilde{U}+\widetilde{A}_{\omega}\widetilde{U}=\chi_{\omega}(x)\widetilde{B}\widetilde{G}, \\
            (\widetilde{U}(0),\widetilde{U}_t(0))=(\widetilde{U}_0,\widetilde{U}_1).
    \end{cases}
 \ee
 then, $U=\Mx\widetilde{U}$ is solution of the following system with control $G=-F\widetilde{U}+M_{u}\widetilde{G}$
 \be
        \begin{cases}
   \label{systemoriginal}
            \pa_t^2 U-\Delta_g U+AU=\chi_{\omega}(x)BG, \\
            (U(0),\partial_t U(0))=\Mx(\widetilde{U}_0,\widetilde{U}_1).
    \end{cases}
\ee
In particular, if the System \eqref{systemBrunov} is controllable in some space $\mathcal{E}=\mathcal{E}_{1}\times \mathcal{E}_{0}$ satisfying $\mathcal{E}_{1}\subset (L^2)^N$ with control $\widetilde{G}\in L^2(0,T;L^2)^K)$, then the System \eqref{systemoriginal} is controllable in $\mathcal{E}_{\Mx}=\Mx \mathcal{E}$ on $[0,T]$ with control $G\in L^2(0,T;L^2)^K)$.

Moreover, we have the following properties concerning the coupling matrix.
\begin{enumerate}
\setlength\topsep{-0.2em}
\setlength\itemsep{-0.2em}
\item \label{itemsubdiag} for any $x\in \mathcal{M}$,  has a subdiagonal form as described in Proposition \ref{propBrunovsky}.
\item \label{itemKalman}for $x\in \omega$, $\widetilde{A}_{\omega}(x)=\widetilde{A}$  so that $(\widetilde{A}_{sub},\widetilde{B})=(\widetilde{A},\widetilde{B})$ satisfies the Kalman rank condition, where the index sub means that we only keep the subdiagonal terms as in \eqref{Asub}.
\end{enumerate}
\end{lem}
\bnp This is just a direct computation, denote $\Box =  \pa_t^2 -\Delta_g$, we have
\bna
\Box U+AU&=& \Mx\Box \widetilde{U}+A \Mx\widetilde{U}=-\Mx \widetilde{A}_{\omega}\widetilde{U}+\chi_{\omega}(x)\Mx\widetilde{B}\widetilde{G}+A \Mx\widetilde{U}\\
&=&-(A\Mx+\chi_{\omega}(x)BF)\widetilde{U}+\chi_{\omega}(x)B M_{u}\widetilde{G}+A \Mx\widetilde{U}\\
&=&\chi_{\omega}(x)B(-F\widetilde{U}+M_{u}\widetilde{G})=\chi_{\omega}(x)BG.
\ena
The second property about the link between the controllability of each equation is direct. The properties of the coupling matrix are then a direct consequences of Proposition \ref{propBrunovsky}.\enp
\subsubsection{An algorithm to obtain a natural subdiagonal form}
In this section, we describe one natural (informal) algorithm as following when considering the control System
\be
    \label{wavezero}
        \begin{cases}
            \pa_t^2 U-\Delta_g U+AU=BG, \\
            (U(0),\partial_t U(0))=(U_0,U_1).
    \end{cases}
    \text{with control $G\in L^2(0,T;L^2)^K)$}
\ee
where $A(x)$ is a matrix in $\R^{N\times N}$ and  $B(x)$ is a matrix in $\R^{N\times K}$.
\medskip

 We start from the subspace of $\R^N$ that might be reached directly by the control $G$ (without using the coupling). Namely, we define $E_1\subset \R^N$ as $
E_1= Vect \left\{\cup_{x\in \mathcal{M}} Range(B(x))\right\}$. This set of state variables might be controlled in $H^1\times L^2$.

 Next, we define the subspace that might be controlled from $E_1$ (if we can control in all $E_1$) through the coupling. This makes us to define naturally $E_2= Vect \left\{\cup_{x\in \mathcal{M}} A(x)(E_1)\right\}$.  This set of state variables might be controlled from a source term in $C([0,T];H^1)$ so, we expect to control the states at least in $H^2(\mathcal{M})\times H^1$ (but may be better for direction that are in $E_1\cap E_2$).

Again, we want to define the subspace that might be controlled from $E_2$ (if we can control in all $E_2$) through the coupling. The natural new space that could be reached is $E_3= Vect \left\{\cup_{x\in \mathcal{M}} A(x)(E_2)\right\}$. This time, the new source term is in $C(0,T;H^2(\mathcal{M}))$. Thus, we expect to control the states in $H^3(\mathcal{M})\times H^2(\mathcal{M})$.

\medskip

So, this leads to the definition of subspaces of $E_i$ by iteration:
\be
\small
E_1= Vect \left\{\cup_{x\in \mathcal{M}} Range(B(x))\right\}; \quad 
E_{i+1}= Vect \left\{\cup_{x\in \mathcal{M}} A(x)(E_i)\right\}.
\ee
$H_k=Vect_{i=1}^k E_i$ is clearly an increasing sequence of subspaces of $\R^N$  that is stationary after some steps that we call $k$. Morever, it satisfies the important property $A(x)(H_i)\subset H_{i+1}$ for any $x\in \mathcal{M}$. It could happen that the bigger space $H_k$ is not equal to $\R^N$, but it is easy to see that the wave system is not controllable in this case. Indeed, for any control $B(x)G(t,x)\in H_k$ since $H_k$ contains $Vect \left\{\cup_{x\in \mathcal{M}} Range(B(x))\right\}$, and $A(x)(H_k)\subset H_{k}$, so for any initial data in $H_k$, the solution remains in $H_k$.

We can then assume now that we can decompose $\R^N= \oplus_{i=1}^k F_i$ with $F_i\cap F_j= \left\{0\right\}$ if $i\neq j$ and $\oplus_{i=1}^n F_i= H_n$. In particular, in a basis according to $F_i$, $A(x),B(x)$ can be written as a matrix "subdiagonal by block" as \eqref{structA} and \eqref{structB}.

Note that in the case $A(x)=A$ and $B(x)=B$, we have $E_1=Range(B)$ and $E_i=Range(A^{i-1}B)$, so that this decomposition is related to the Kalman rank condition and the Brunovsky normal form described in Proposition \ref{propBrunovsky}.
\subsection{Wellposedness in multilevel spaces}
Up to now and in the next Section, we assume that $A(x)$ and $B(x)$ have the form described in Theorem \ref{thmcontrolmulti}.

The natural space for solutions of \eqref{wavezero} is then then space $\mathcal{H}^s$ as follows.
$U\in \mathcal{H}^s$ if for every $i=1,...,k$, we have $U^i \in (H^{s+i-1})^{d_i}$ where $d_i$ is the dimension of $F_i$. That is
\be\lab{zero-Hs-space}
\mathcal{H}^s= (H^{s})^{d_1}\times (H^{s+1})^{d_2}\times \cdots \times H^{s+k-1}(\mathcal{M})^{d_k}.
\ee
The natural energy space is then $\mathcal{E}=\mathcal{H}^1\times \mathcal{H}^0$ and we will prove (see Theorem \ref{thmwellposedmulti}) that the equation
\be
    \label{controlzerointro}
        \begin{cases}
            \pa_t^2 U-\Delta_g U+AU=BG, \\
            (U(0),\partial_t U(0))=(U_0,U_1).
    \end{cases}
\ee
is well posed in $\mathcal{E}$ with source term $G\in L^2(0,T;L^2)^K)$.

Now, it appears that the important terms are the subdiagonal terms of $A$ as \eqref{structAtilde}.
\be \small
\label{structAtilde}
A_{sub}(x)=\begin{bmatrix}
    0 &  \dots  &\dots  &0 \\
    A_{21} &  \dots  &\dots  & 0 \\
    \vdots &  \ddots &\ddots  & \vdots \\
    0 & 0& A_{k,k-1}  & 0
\end{bmatrix}
\ee

Note that in the previous result, the high frequency problem and the unique continuation problem, the matrix involved is not the same. We have
\be\small
\label{decompA}
A(x)=A_{sub}+A_r=\begin{bmatrix}
    0 & \dots  &\dots  &0 \\
    A_{21} &  \dots  &\dots  & 0 \\
    \vdots &  \ddots &\ddots  & \vdots \\
    0 & 0& A_{k,k-1}  & 0
\end{bmatrix}+\begin{bmatrix}
    A_{11} & A_{12}   &\dots  & A_{1k} \\
   0 & A_{22}   &\dots  & A_{2k} \\
    \vdots & \vdots & \ddots  & \vdots \\
    0 & 0& 0   & A_{kk}
\end{bmatrix}
\ee

 The structure "subdiagonal by block " of $A$ allows to prove the following Lemma.
\begin{lem}
\label{lmAhstordu}
For any $s\in \R$, the multiplication by \begin{itemize}
\setlength\itemsep{-0.2em}
\item $A(x)$ sends $\mathcal{H}^s$ into $\mathcal{H}^{s-1}$
\item $A_{sub}(x)$ sends $\mathcal{H}^s$ into $\mathcal{H}^{s-1}$
\item $A_r(x)$ sends $\mathcal{H}^s$ into $\mathcal{H}^{s}$.
\end{itemize}
\end{lem}
\begin{lem}
\label{lmexistsource}
Let $(U_0,U_1)\in \mathcal{E}$ and $H\in L^1([0,T];\mathcal{H}^0)$. Then, there exists a unique solution $(U,\partial_t U)\in C([-T,T],\mathcal{E})$ to
\be
    \label{eqsource }
        \begin{cases}
            \pa_t^2 U-\Delta_g U=H, \\
            (U(0),\partial_t U(0))=(U_0,U_1).
    \end{cases}
\ee
is well-posed in $(U,\partial_t U)\in C([-T,T],\mathcal{E})$ for $(U_0,U_1)\in \mathcal{E}$ and $H\in L^1([0,T];\mathcal{H}^0)$
\end{lem}
\bnp
Since the wave operator is diagonal, we can reduce the problem to each component where the Theorem reduces to the property that the equation
\be
\nonumber
        \begin{cases}
            \pa_t^2 V_i-\Delta_g V_i=H_i, \\
            (V_i(0),\partial_t V_i(0))=(U_{0,i},U_{1,i})
    \end{cases}
\ee
is well-posed in $C([-T,T],(H^{i}(\mathcal{M}))^{d_i}\cap C^1([0,T];(H^{i-1}(\mathcal{M}))^{d_i}) $ with source term in $H_i\in L^1([0,T];(H^{i-1}(\mathcal{M}))^{d_i})$ and $(U_{0,i},U_{1,i})\in (H^{i}(\mathcal{M})\times H^{i-1}(\mathcal{M}))^{d_i}$.
\enp
\begin{thm}
\label{thmwellposedmulti}
Let $(U_0,U_1)\in \mathcal{E}$ and $G\in L^1([0,T];(L^2)^K)$. Then, there exists a unique solution $(U,\partial_t U)\in C([-T,T],\mathcal{E})$ to the equation
\be
\nonumber
        \begin{cases}
            \pa_t^2 U-\Delta_g U+AU=BG, \\
            (U(0),\partial_t U(0))=(U_0,U_1).
    \end{cases}
\ee
\end{thm}
\bnp
The proof is direct with Lemma \ref{lmAhstordu} and Lemma \ref{lmexistsource}. The source term $BG$ is in $L^1([0,T];\mathcal{H}^0)$ because of the specific structure of $B$ in \eqref{structB}.
\enp
\subsection{Reduction of the control problem}
In this Section, we will reduce the control problem, which is now with a coupling of subdiagonal form as in Section \ref{sectzeroorder}, to a coupling of order $1$. This will lead to a proof of Theorem \ref{thmcontrolmulti}

At this stage, we notice that the matrix $A_r$ defined in \eqref{decompA} is compact for this scale of spaces. Now, it is natural to define the following operator
\be\small
T=\begin{bmatrix}
    Id  & 0 & \dots  & 0 \\
    0& \Lambda & \dots  & 0 \\
    \vdots & \vdots & \ddots  & \vdots \\
    0 & 0& 0  & \Lambda^{k-1}
\end{bmatrix}
\ee
$T$ is a natural isometry from $\mathcal{H}^s$ to $(H^s)^N$.
We will need also the matricial operator $P_A$ defined by (roughly the action of $P_A$ on each subspace is described by $(P_A)_{i,j}=\Lambda^{i-1}A_{i,j}\Lambda^{-(j-1)})$
\be \small
\nonumber
P_A=TA T^{-1}=\begin{bmatrix}
    A_{11} & A_{12}\Lambda^{-1} & A_{13}\Lambda^{-2} & \dots  &\dots  & A_{1k}\Lambda^{-(k-1)} \\
    \Lambda A_{21} & \Lambda A_{22}\Lambda^{-1} & \Lambda A_{23}\Lambda^{-2} & \dots  &\dots  & \Lambda A_{2k}\Lambda^{-(k-1)} \\
		0 &\Lambda^2 A_{32}\Lambda^{-1} &\Lambda^2 A_{33}\Lambda^{-2} & \dots  &\dots  &\Lambda^2 A_{3k}\Lambda^{-(k-1)} \\
    \vdots & \vdots & \vdots & \ddots &\ddots  & \vdots \\
    0 & 0& 0 & 0  &\Lambda^{k-1}A_{k,k-1}\Lambda^{-(k-2)}  & \Lambda^{k-1}A_{kk}\Lambda^{-(k-1)}
\end{bmatrix}
\ee
Therefore, we have the immediate property.
\begin{lem}
\label{lmsymbolePA}
$P_A$ is a pseudodifferential operator of order $1$ of principal symbol $\lambda(x,\xi) A_{sub}(x) $.
\end{lem}
Also, the following Lemma is immediate noting that $TB=B$.
\begin{lem}
\label{lmchangeunknown}
Let $(U_0,U_1)\in \mathcal{E}$ and $G\in L^2([0,T];(L^2)^K)$. Then, the following statements are equivalent
\begin{enumerate}
\item $(U,\partial_t U)\in C([-T,T],\mathcal{E})$ is solution to the equation
\be
    \nonumber
        \begin{cases}
            \pa_t^2 U-\Delta_g U+AU=BG, \\
            (U(0),\partial_t U(0))=(U_0,U_1).
    \end{cases}
\ee
\item $(V,\partial_t V)=T(U,\partial_t U)\in C([-T,T],(H^1\times L^2))^N)$ is solution to the equation
\be
    \nonumber
        \begin{cases}
            \pa_t^2 V-\Delta_g V+P_A V=BG, \\
            (V(0),V_t(0))=T(U_0,U_1).
    \end{cases}
\ee
\item $(W,\partial_t W)=\Lambda T(U,\partial_t U)\in C([-T,T],(L^2)\times H^{-1})^N)$ is solution to the equation
\be
    \label{eqnW}
        \begin{cases}
            \pa_t^2 W-\Delta_g W+\Lambda P_A \Lambda^{-1}W=\Lambda B G, \\
            (W(0),W_t(0))=\Lambda T(U_0,U_1).
    \end{cases}
\ee
\end{enumerate}
\end{lem}
\begin{prop}[HUM]\label{propequivzero}The following statements are equivalent
\begin{enumerate}
\item \label{itemcontrol}The  problem \eqref{eqnW} is controllable in $(L^2\times H^{-1})^N$ with control $G\in L^2([0,T];L^2)^K)$
\item \label{itemobserv} We have the observability estimate
\be
   \nonumber
    C^1_{obs} \int_0^T\|B^* \Lambda W\|^2_{(L^2)^K}dt \geq \mathbb{E}_0(W_0,W_1),
\ee
for any solution to
\be
 \nonumber
        \begin{cases}
            \pa_t^2 W-\Delta_g W+\Lambda^{-1} P_A^* \Lambda W=0, \\
            (W(0),W_t(0))=(W_0,W_1).
    \end{cases}
\ee
\item \label{itemtruc1}\eqref{mainmainode1} is controllable and for any $\lambda\in \C$, any solution $W\in (H^1)^N$ of
\be
\nonumber
    \begin{cases}
     -\Delta_g W+\lambda^2 W+ \Lambda^{-1} P_A^* \Lambda W=0,\\
     B^* \Lambda W=0,
    \end{cases}
\ee
is $V=0$.
\item \label{itemtruc2}\eqref{mainmainode1} is controllable and for any $\lambda\in \C$, any solution $U\in (H^1)^N$ of
\be
\nonumber
    \begin{cases}
     -\Delta_g U+\lambda^2 U+ A^*  U=0,\\
      B^* U=0,
    \end{cases}
\ee
is $V=0$.
\end{enumerate}
\end{prop}
\bnp
\ref{itemcontrol} $\Leftrightarrow $\ref{itemobserv} is exactly the classical HUM method. We refer for instance to \cite{Coron07}.

\ref{itemobserv}$\Leftrightarrow $\ref{itemtruc1} is exactly Theorem \ref{mainthmstrong} once we have noticed that $\Lambda^{-1} P_A^* \Lambda$ is a pseudodifferential of order $1$ with principal symbol $\lambda(x,\xi) A_{sub}(x)^*$ as noticed in Lemma \ref{lmsymbolePA}, while $B^* \Lambda$ is of symbol $\lambda(x,\xi) B(x)^* $. Note also that $P_A$ and $\lambda(x,\xi) B(x)^* $ are not differential operators, but Theorem \ref{mainthmPseudo} is still true and we can apply Proposition \ref{propodexxxx} to get the same result, using that $\lambda(x,\xi)$ is even in $\xi$.

\ref{itemtruc1}$\Leftrightarrow $\ref{itemtruc2} is obtained undoing the change of variable done in Lemma \ref{lmchangeunknown} in the elliptic equation (modulo some duality). More precisely, $W$ solves the equation $-\Delta_g W+\lambda^2 W+ \Lambda^{-1} P_A^* \Lambda W=0$ if and only if $U=T\Lambda W$ solves $-\Delta_g U+\lambda^2 U+ A^*(x)U=0$. $B^* \Lambda W=0$ is equivalent to $B^* T^{-1}U=0$ and then $B^*U=0$ since $T^{-1}B=B$ and so $B^*=B^*T^{-1}$.
\enp
Theorem \ref{thmcontrolmulti} follows then as a combination of Lemma \ref{lmchangeunknown} and Proposition \ref{propequivzero}.
\section{Examples}\lab{sectionexample}

In this section, we provide two examples as applications of Theorem \ref{mainthm-w}. We will treat the wave equations coupled by velocities of Cascade type, and the wave equations coupled by velocities with (almost) constant coefficients. The results are not always new, but the proof we provide has the advantage to always rely on easy ODE analysis which, we believe makes it valuable and give a common feature for this systems studied in different articles.
\subsection{Wave equations coupled by velocities of cascade-type}

We first consider the observability problem for
wave system Coupled by Velocities of cascade-type:
\be
\lab{toy model two}
\begin{cases}
        \partial_{t}^2u -\Delta_g u +u+ \beta(t,x) \pa_t v    =   0, \\
        \pa_{t}^2 v-\Delta_g v +v= 0,
\end{cases}
\ee
where the coupling term $\beta\in C^\infty([0,T]\times \mathcal{M}) $.

    Based on Theorem \ref{mainthm-w}, we can prove the following statement. The result is mostly contained in \cite{DehmanLeLeautaud14} which considers the same problem with zero order coupling or coupling $\beta(t,x) \Lambda v$   for which the analysis is almost the same. Yet, we believe that the proof we present here, which mostly relies on Theorem \ref{mainthm-w} and ODE analysis, is interesting because it gives some ODE interpretation of some computations that were performed in \cite{DehmanLeLeautaud14}. We refer for example to \cite[Theorem 5.3]{DehmanLeLeautaud14} where the matrix of the principal symbol of the HUM operator is computed and corresponds to the Gramian operator of the ODE control problem that we compute in Lemma \ref{lmequivODEcascade} below.
\begin{prop}\lab{example:cascade:prop1}
 Assume that $\alpha\in C^\infty(\R\times \mathcal{M})$. Then weak observability inequality \be\lab{weakly obs one}\begin{split}
\int_0^T \int_{\mathcal{M}} \alpha^{2} (|\nabla u|_g^2+|u|^{2})dxdt&+
c\|(u_0,u_1,v_0,v_1)\|^2_{H^{\f{1}{2}}\times H^{-\f{1}{2}}\times H^{\f{1}{2}}\times H^{-\f{1}{2}}}
\\&\geq C \|(u_0,u_1,v_0,v_1)\|^2_{H^1\times L^2\times H^1\times L^2},
\end{split}
\ee
holds  if and only if $\alpha,\beta$ satisfy the following property
\bnan
\nonumber
&&\forall \rho_0\in S^*\mathcal{M}, \exists 0<t_1<t_2< T, \textnormal{ such that }\\
 \label{Tint}
 &&  \quad \alpha(t_{1},\varphi_{t_1}(\rho_0))\neq 0,\alpha(t_{2},\varphi_{t_2}(\rho_0))\neq 0, \int_{t_{1}}^{t_{2}}\beta(\tau,\varphi_{\tau}(\rho_0))\neq 0
\enan
Here $\varphi_t$ is Hamiltonnian flow of $|\xi|_x$ defined in Theorem \ref{mainthm-w}, $c,C$ are two positive constants independent of the initial data.
\end{prop}

\bnp[Proof of Proposition \ref{example:cascade:prop1}]
We apply Theorem \ref{mainthm-w} (actually a variant) with $D(u,v)=\alpha(t,x) \Lambda u$ with $\Lambda=(-\Delta_g+1)^{1/2}$ and $L(u,v)=(\beta(t,x)v_{t},0)$, which states that the weak observability is equivalent to the controllability of the following ODE system for any $\rho_0\in S^{*}M$:
\be \small
\lab{toy model two ode}
    \begin{cases} \displaystyle
       \dot{ X}(t)=\f{-\beta(t,\varphi_t(\rho_0))}{2}  \left(
                     \begin{array}{cc}
                       0 & 0 \\
                       1 & 0 \\
                     \end{array}
                   \right)
        X(t)+ \f{\alpha(t,\varphi_t(\rho_0))}{2} \left(
             \begin{array}{c}
               1  \\
               0 \\
             \end{array}
           \right)g(t),\\
    X(0)=X_0\in \R^N.
    \end{cases}
\ee
where $g\in L^2(0,T)$ is a scalar control function.
The proposition follows then directly from Lemma \ref{lmequivODEcascade} below.
\enp

Under additional assumptions, we can obtain the strong observability, as in \cite{DehmanLeLeautaud14}.
\begin{prop}
\label{prop:cascadepos} With the assumptions as Proposition \ref{example:cascade:prop1},  let us assume furthermore that $\alpha$ and $\beta$ only depend on $x$ and $\beta$ satisfies sign condition, i.e., $\beta\geq0$ (or $\beta\leq 0$), then  the observability inequality
\be\lab{ex:1:strong obs onecascade}
\int_0^T \int_{\mathcal{M}}\alpha^{2} (|\nabla u|_g^2+|u|^{2})dxdt\geq C \|(u_0,u_1,v_0,v_1)\|^2_{H^1\times L^2\times H^1\times L^2},
\ee
holds  if and only if $T> T_{\omega\rightarrow o\rightarrow \omega}$,
where $T_{\omega\rightarrow o\rightarrow \omega}$ (cf.\cite{DehmanLeLeautaud14}) is defined by
 \be\displaystyle\lab{cascade time}\begin{split}
T_{\omega\rightarrow o\rightarrow \omega}=\inf\{&T>0  ~s.t.~ \forall \varphi_0(\rho_0)=\rho_0\in S^*\mathcal{M}, \exists 0<t_1<t_2<t_3< T ,\\ & \quad\text{such that}~ \alpha(\varphi_{t_1}(\rho_0)),\alpha(\varphi_{t_3}(\rho_0))\neq 0,\beta(\varphi_{t_2}(\rho_0))\neq 0\}.
\end{split}
\ee
\end{prop}

\bnp[Proof of Proposition \ref{prop:cascadepos}]
   We apply Lemma \ref{lmequivODEcascade} (the case $\beta\geq0$) to get the equivalence for weak observability, Following Theorem \ref{mainthmstrong}, it only suffices to prove System \eqref{toy model two} satisfies unique continuation. Let $A=\left(
                         \begin{array}{cc}
                           0 & 0 \\
                           1 & 0 \\
                         \end{array}
                       \right), B=\left(
                                    \begin{array}{c}
                                      1 \\
                                      0 \\
                                    \end{array}
                                  \right),
  $  it is easy to see that $A,B$ satisfy Kalman Rank Condition and $A$ only has eigenvalue 0. By Proposition \ref{eigenproblem:prop} in the Appendix, we conclude the proof of unique continuation of System \eqref{toy model two} and therefore of the Proposition. Note that Proposition \ref{eigenproblem:prop} does not take into account the case $\lambda=0$ in \eqref{eqnUCP}. Yet, this case is trivial because we have replaced the wave equation by the Klein-Gordon. Indeed, $(u,v)$ is solution of $0=-\Delta_gu+u=-\Delta_gv +v$ and is zero.
	\enp
\begin{lem}
\label{lmequivODEcascade}
We have the following equivalence for $\alpha$, $\beta\in C([0,T])$:
\begin{enumerate}
\setlength\itemsep{-0.2em}
\item \label{enumODE}The following control system is controlable. \be
    \begin{cases} \displaystyle
       \dot{ X}=\beta(t)  \left(
                     \begin{array}{cc}
                       0 & 0 \\
                       1 & 0 \\
                     \end{array}
                   \right)
        X+ \alpha(t) \left(
             \begin{array}{c}
               1  \\
               0 \\
             \end{array}
           \right)g,\\
    X(0)=X_0\in \R^2,
    \end{cases}
\ee
where $g\in L^2(0,T)$ is a scalar control function.
\item \label{enum2prop} There is no $(c,d)\in \C^{2}\setminus (0,0)$ so that $c\alpha(t)=d \alpha(t)\int_0^{t}\beta(\tau)d\tau$ for all $t\in [0,T]$.

\item \label{enumt123}
There exists $0<t_1<t_2< T$, such that
\bna
 \alpha(t_{1})\neq 0,\alpha(t_{2})\neq 0, \int_{t_{1}}^{t_{2}}\beta(\tau)\neq 0.
\ena

\end{enumerate}
Moreover, if in addition, we have $\beta(t)\geq0$ (or $\beta(t)\leq0$), this is also equivalent to
\bna \exists 0<t_1<t_2<t_{3}< T, \textnormal{ such that } \alpha(t_{1})\neq 0,\alpha(t_{3})\neq 0, \beta(t_{2})\neq 0.
\ena
\end{lem}
\bnp
\ref{enumODE}$\Leftrightarrow $\ref{enum2prop} follows from classical control theory of finite dimensional system, we omit it.

Now, we prove \ref{enumt123}$\Rightarrow $\ref{enum2prop}.
Assume
$t_1,t_2 \text{ such that } 0< t_1<t_2< T, \alpha(t_1),\alpha(t_2)\neq 0,  \int_{t_{1}}^{t_{2}}\beta(\tau)\neq 0
.$
Take $(c,d)$ so that $c\alpha(t)=d \alpha(t)\int_t^{0}\beta(\tau)d\tau$ for all $t\in [0,T]$, we shall prove $c=d=0$. We have then, since $\alpha(t_{1})\neq 0$ and $\alpha(t_{2})\neq 0$
\bnan
\label{eqnt12}
c=d \int_{t_{1}}^{0}\beta(\tau)d\tau; \quad c=d\int_{t_{2}}^{0}\beta(\tau)d\tau
\enan
and by difference $0=d \int_{t_{1}}^{t_{2}}\beta(\tau)d\tau$, so $d=0$ since $\int_{t_{1}}^{t_{2}}\beta(\tau)d\tau\neq 0$. This gives $c=0$ after \eqref{eqnt12}.

We finish with \ref{enum2prop}$\Rightarrow $\ref{enumt123}.

First, \ref{enum2prop} implies that $\alpha\not\equiv 0$ (otherwise any $(c,d)\neq 0$ works)) and there exists $t_{1}$ so that $\alpha(t_{1})\neq 0$. Define the function $f(t)=\alpha(t)\int^{t}_{t_{1}}  \beta(\tau)d\tau$. We prove $f\not\equiv 0$. Indeed, if it is the case, we have $0= \alpha(t)\int^{t}_{t_{1}} \beta(\tau)d\tau$ for all $t\in[0,T]$. In particular, $0=\alpha(t)\left[\int_{t}^{0}  \beta(\tau)d\tau-\int_{t_{1}}^{0}  \beta(\tau)d\tau\right]$ for all $t\in[0,T]$, which is impossible by assumption. So, we have proved $f\not\equiv0$ and there exists $t_{2}$ with $f(t_{2})\neq0$, and in particular, $\alpha(t_{2})\neq 0$ and $\int^{t_{2}}_{t_{1}}  \beta(\tau)d\tau\neq 0$, which is the expected property \ref{enumt123}, up to exchanging the role of $t_{1}$ and $t_{2}$. Note that we have only selected $0\leq t_{1}<t_{2}\leq T$, but we can impose strict inequality with the same conclusion by continuity.

The last equivalence if $\beta\geq0$ is obvious. \enp

\subsection{Wave equations coupled by first or zero order terms of constant coefficients}
\label{sectioncouplingconstant}
In this Section, we explain how our result allows to recover and precise some results of Liard-Lissy \cite{LiardLissy17} and Lissy-Zuazua \cite{LissyZuazua} which were obtained with a complete different method. In particular, it allows to precise the regularity of the directions that can be reached.

Let $A\in \R^{ N\times N}$ and $B\in\R^{ K\times N}$ be constant matrices. In the notations of Section \ref{sectzeroorder}, we place ourselves in the particular cases: $A(x)=A$ constant and $B(x)=B \chi_{\omega}$

In particular, our results precise the result in the following sense. \cite[Theorem 4.2]{LiardLissy17} proves controllability in $(H^{2N-1})^N\times (H^{2N-2})^N$ with control in $L^2)$ under the Kalman Rank Condition, i.e: \be\lab{kalman rank condition}
rank(B,AB,...,A^{N-1}B)=N.
\ee
Our results proves the same result in $\mathcal{H}^1\times \mathcal{H}^0$ which is defined by \eqref{zero-Hs-space}.

Two situations can be considered, coupling of order $1$ or $0$, that we detail in separate subsection.
\subsubsection{Constant coupling of order $1$}
We consider the following system of wave equations on a compact manifold ($\mathcal{M},g$):
\be\lab{cons-case}
\begin{cases}
(\pa_t^2-\Delta_g +1)V+A\pa_tV=B\chi_\omega(x) u.\\
(V(0),\partial_tV(0))=(V_0,V_1).
\end{cases}
\ee
where $V\in \mathbb{R}^N$, $A\in \R^{N\times N}$ and $B\in \R^{N\times K}$, $u\in L^2(0,T; (L^2)^K)$. $\chi_\omega(x)$ denotes a smooth  function  which satisfies
\be\lab{characteristicfun}
\chi_\omega(x):= \begin{cases}
1, \quad \text{if }~x\in \omega;\\
0,\quad \text{if }~x\in \mathcal{M}\backslash\tilde{\omega}
\end{cases}
\ee
where $\omega\subset \tilde{\omega}$.
Weak solution of \eqref{cons-case} exists with initial data $(V_0,V_1)\in  (L^{2})^N \times (H^{-1})^N$.

\begin{prop}\lab{cons-case:prop}
  Assume $A,B$ satisfy Kalman rank condition and $\omega$ satisfies GCC. Then System \eqref{cons-case} is exactly controllable with initial data $(V_0,V_1)\in  (L^{2})^N \times (H^{-1})^N$.
  \end{prop}
\bnp[Proof of Proposition \ref{cons-case:prop}] Firstly, we will apply Corollary \ref{maincorfc} (actually a variant) with $D^* u =B\chi_\omega (x) u$ and $L^* V=A \partial_t V$, which states that co-dimensional controllability (the weak observability of dual system) is equivalent to the controllability of the following ODE system for any $\rho\in S^{*}M$:
\be\lab{cons-case:ode}
\begin{cases}
\dot X(t)=\f{1}{2}AX(t)+\f{1}{2}B\chi_\omega (\phi_t(\rho_0))u,\\
X(0)=X_0\in \R^N.
\end{cases}
\ee
Since $\omega$ satisfies GCC,$\forall \rho_0\in S^*\mathcal{M}$, we can find an interval $[t_1,t_2]\in \R$, such that $\chi_\omega(\phi_t(\rho_0)) =1, \forall t\in [t_1,t_2]$.
Hence we obtain the exact controllability of \eqref{cons-case:ode} following from classical control theory of ode.
Next we  only need to show the unique continuation property of the following elliptic equations:
 \be\lab{cons-case:ucp}
 \begin{cases}
 (\lambda^2-\Delta_g +1)v-A^{tr}\lambda v=0;\\
 B^{tr}\chi_\omega v=0
 \end{cases} \Rightarrow v\equiv 0.
 \ee
Since $\omega\cap \mathcal{M}=\omega\subset \mathcal{M}$ and $A,B$ satisfy Kalman rank condition, by using Proposition \ref{eigenproblem:prop}, we conclude our proposition \ref{cons-case:prop}.
\enp
\subsubsection{Constant coupling of order $0$}
We consider the controllability of the system of wave equations coupled in order zero:
\be\lab{cons-case2}
\begin{cases}
(\pa_t^2-\Delta_g )V+AV=B\chi_\omega(x) u.\\
(V(0),\partial_tV(0))=(V_0,V_1).
\end{cases}
\ee
where  $V\in \mathbb{R}^N$, $A\in \R^{N\times N}$ can be written  a matrix "subdiagonal by block" as \eqref{structA} and $B\in \R^{N\times K}$ can be written as \eqref{structB}, $u\in L^2(0,T;(L^2)^K)$, $\chi_\omega$ satisfies \eqref{characteristicfun}.
\begin{prop}
\label{cons-case2-prop}
Assume that $A,B$ satisfy the Kalman rank condition and $\omega$ satisfies GCC. Then, with the notations of Lemma \ref{lmequivBrunov}, System \eqref{systemBrunov} is controllable in the space $\mathcal{E}=\mathcal{H}^1\times \mathcal{H}^0$ defined by $\mathcal{H}^s= (H^{s})^{d_1}\times (H^{s+1})^{d_2}\times \cdots \times H^{s+k-1}(\mathcal{M})^{d_k}$ where $k$ and $d_{i}\in \N$, $i=1,\cdots, k$ are given by Proposition \ref{propBrunovsky}.
\end{prop}
\bnp
We want to apply Theorem \ref{thmcontrolmulti}. First of all, by Item \ref{itemsubdiag} of Lemma \ref{lmequivBrunov}, the matrix $\widetilde{A}_{\omega}$ satisfies the subdiagonal condition with respect to the splitting of the variables defined by the $d_{i}$. This gives also that System \eqref{systemBrunov} is well posed following from Theorem \ref{thmwellposedmulti}. Then by using Theorem \ref{thmcontrolmulti}, we only need to show the unique continuation of eigenfunctions and the controllability of the following ODE system:
 \be
    \lab{cons-case2-ode1}
\left\{ \begin{array}{lll}
        \displaystyle \dot{ X}(t)=\f{1}{2}\widetilde{A}_{\omega}(\varphi_{ t}(\rho_0)) X(t)+\f{1}{2}{\widetilde{B}\chi_{\omega}(\varphi_{ t}(\rho_0))u(t)},  \\
        X(0)=X_0\in \R^N.
    \end{array}\right.
 \ee
 Item \ref{itemKalman} of Lemma \ref{lmequivBrunov} ensures that for $x\in \omega$, $\widetilde{A}_{\omega}(x),\widetilde{B}\chi_{\omega}(x))$ satisfy Kalman rank condition. Since $\omega$ satisfies GCC, this means that for any $\rho_{0}\in S^{*}\mathcal{M}$, there exists $t\in [0,T]$ so that $\pi_{x}\varphi_{ t}(\rho_0)\in \omega$ and therefore $\widetilde{A}_{\omega}(\varphi_{ t}(\rho_0)),\widetilde{B}\chi_{\omega}(\varphi_{ t}(\rho_0)))$ satisfy Kalman rank condition. Hence the System \eqref{cons-case2-ode1} is controllable.

Concerning the unique continuation of eigenfunctions, we notice that if $\widetilde{U}$ is solution to
\be
\nonumber
        \begin{cases}
            -\Delta_g \widetilde{U}+\widetilde{A}_{\omega}^*\widetilde{U}=\lambda \widetilde{U}, \\
           \chi_{\omega}(x)\widetilde{B}^*\widetilde{U}=0.
    \end{cases}
 \ee
 then, in fact $-\Delta_g \widetilde{U}+M_x^*A^*(M_x^{-1})^*\widetilde{U}=\lambda \widetilde{U}$ and $U=(M_x^{-1})^*\widetilde{U}$ is solution to
 \be
\nonumber
        \begin{cases}
            -\Delta_g U+A^*U=\lambda U, \\
          \chi_{\omega}(x)B^*U=0.
    \end{cases}
\ee
for which we can apply Proposition \ref{eigenproblem:prop}. So we finish the proof of Proposition \ref{cons-case2-prop}.
\enp
A combination of Lemma \ref{lmequivBrunov} and Proposition \ref{cons-case2-prop} simply concludes Theorem \ref{cons-case-propintro}.
\appendix

\section{Appendix}
\subsection{Control problem of finite dimensional system}\lab{section1.2.3}
In this Section, we recall well known facts about the control of finite dimensional systems. We refer to \cite{Coron07} for more details.

The following Proposition is a reformulated and precised version of the Brunovsky normal form \cite{Brunovsky} for control of ODE. We provide a proof of it because we did not find it written in this way and we needed a slight modification with a matrix $\widetilde{A}_t$ which will be useful in the change of variable of Lemma \ref{lmequivBrunov}. Yet, it is quite classical in control theory, and we don't claim novelty, see for instance \cite{Trelatbook}.
\begin{prop}[Brunovsky normal form]
\label{propBrunovsky}
Assume ~$A\in \R^{N\times N}, B\in \R^{N\times K}$~ satisfy the Kalman rank condition
and denote $m=rank (B)$. Then, there exist some matrices $\Mx\in GL_{N}(\R)$, $M_{u}\in GL_{K}(\R)$ and $F \in \R^{K\times N}$, and some nonincreasing sequence of integers $d_{i}, i=1,\cdots, k$ (for $k\leq n$) so that
\bna
\widetilde{A}= \Mx^{-1}(A\Mx+BF);\quad \widetilde{B}=\Mx^{-1}BM_{u}
\ena
with
\be \small
\label{AtildeKalman}
\widetilde{A}=\begin{bmatrix}
    0 &  \dots  &\dots  &0 \\
    A_{21}  & \dots  &\dots  & 0 \\
    \vdots   & \ddots &\ddots  & \vdots \\
    0 & 0   &A_{k,k-1}  & 0
\end{bmatrix}
;\quad
\widetilde{B}=\begin{bmatrix}
   Id_{m}& 0_{m,K-m}\\
   0_{N-m,m}&0_{N-m,K-m}
   \end{bmatrix}.
\ee
with $A_{i+1,i}= \begin{bmatrix}
   Id_{d_{i+1}} & 0_{d_{i+1},d_{i}-d_{i+1}}   \\
    0_{d_{i}-d_{i+1},d_{i+1}} &     0_{d_{i}-d_{i+1},d_{i}-d_{i+1}}
	   \end{bmatrix}\in \R^{d_{i+1} \times d_i} \ (i=1,\cdots,k)$, that is $A_{i+1,i}(k,l)=\delta_{k,l}$ (recall that $d_{i+1} \leq d_i$).
	
	   Moreover, $(\widetilde{A}, \widetilde{B})$ also satisfy the Kalman rank condition.
	
	   Also, for any $t\in \R$, we also have the following form for $\widetilde{A}_{t}= \Mx^{-1}(A\Mx+tBF)$,
		\bna \small
\widetilde{A}_{t}=\begin{bmatrix}
    * &  \dots  &\dots  &* \\
    A_{21} &  \dots  &\dots  & 0 \\
    \vdots &   \ddots &\ddots  & \vdots \\
    0 &  0  &A_{k,k-1}  & 0
\end{bmatrix}.
\ena
\end{prop}
\bnp We prove the result by iteration on the dimension. The initialization is trivial, so we prove the iteration.

There exists $\Mx_{1}$, $M_{u,1}$ so that $\Mx_{1}^{-1}BM_{u,1}=\begin{bmatrix}
   Id_{m}& 0_{m,K-m}\\
   0_{N-m,m}&0_{N-m,K-m}
   \end{bmatrix}.
$ We define
\bna
C=\Mx_{1}^{-1}A\Mx_{1}=\begin{bmatrix}
    C_{1,1} & C_{1,2} \\
    C_{2,1} & C_{2,2}
   \end{bmatrix},
\ena
where $C_{1,1}\in \R^{m\times m}$, $C_{2,2} \in \R^{(N-m)\times (N-m)}$. Since $A,B$ satisfies Kalman rank condition, it is easy to obtain that $C,\Mx_{1}^{-1}BM_{u,1}$ satisfies Kalman rank condition. By Hautus Lemma, we check that $(C_{2,2}, C_{2,1})$ satisfies the Kalman rank condition. Indeed,
\be
rank(\lambda-C,\Mx_{1}^{-1}BM_{u,1})=N, \quad  \forall \lambda\in \C,
\ee
which is equivalent to
\be
rank(\begin{bmatrix}
    \lambda-C_{1,1} & -C_{1,2}
 &  Id_{m}& 0_{m,K-m}\\
       -C_{2,1} & \lambda-C_{2,2}&
   0_{N-m,m}&0_{N-m,K-m}
   \end{bmatrix})=N, \quad   \forall \lambda\in \C,
\ee
so that
\be
rank(\begin{bmatrix}
    -C_{2,1} & \lambda-C_{2,2}
   \end{bmatrix})=N-m, \quad   \forall \lambda\in \C,
\ee
then we obtain $(C_{2,2}, C_{2,1})$ satisfies the Kalman rank condition.
By iteration, there exists $G_{x}\in GL_{N-m}(\R)$ and $G_{u}\in GL_{m}(\R)$ and $F_{2}\in \R^{m\times (N-m)}$ so that
\bna
\widetilde{A}_{N-m}= G_{x}^{-1}(C_{2,2}G_{x}+C_{2,1}F_{2});\quad \widetilde{B}_{N-m}=G_{x}^{-1}C_{2,1}G_{u}.
\ena
has the expected form. We define
\bna
\Mx_{2}=\begin{bmatrix}
    Id_{m}& F_{2} \\
     0_{N-m,m} &  G_{x}
   \end{bmatrix}; \quad \Mx_{2}^{-1}=\begin{bmatrix}
    Id_{m}& *\\
   0_{N-m,m} &  G_{x}^{-1}
   \end{bmatrix},
\ena
\bna
\Mx_{2}^{-1}C\Mx_{2}=\begin{bmatrix}
   *& *\\
     G_{x}^{-1}C_{2,1}&G_{x}^{-1}C_{2,1}F_{2}+G_{x}^{-1}C_{2,2}G_{x}
   \end{bmatrix}= \begin{bmatrix}
   *& *\\
     \widetilde{B}_{N-m}G_{u}^{-1}&\widetilde{A}_{N-m}
   \end{bmatrix}.\ena
Now, we define
\bna
\Mx_{3}=\begin{bmatrix}
    G_{u} & 0 \\
     0_{N-m,m} &  Id_{N-m}
   \end{bmatrix}; \quad \Mx_{3}^{-1}=\begin{bmatrix}
 G_{u}^{-1}& 0\\
   0_{N-m,m} & Id_{N-m}
   \end{bmatrix}
\ena
so that for $\Mx=\Mx_{1}\Mx_{2}\Mx_{3}\in GL_{N}(\R)$, we have
\bna
\Mx^{-1}A\Mx= \begin{bmatrix}
  T_{1}& T_{2}\\
     \widetilde{B}_{N-m}&\widetilde{A}_{N-m}
   \end{bmatrix};\quad \Mx^{-1}BM_{u,1}= \begin{bmatrix}
   G_{u}^{-1}& 0\\
   0&0
   \end{bmatrix};
\ena
for some matrix $T_{1}\in \R^{m\times m}$ and $T_{2}\in \R^{m\times (N-m)}$.
So, choosing finally
\bna
M_{u}=M_{u,1} \begin{bmatrix}
 G_{u}& 0\\
   0&Id_{K-m}
   \end{bmatrix};\quad F=-M_{u}\begin{bmatrix}
  T_{1}& T_{2}\\
   0_{K-m,m}&0_{K-m,N-m}
   \end{bmatrix}
   \ena
   we get
 \bna
\Mx^{-1}(A\Mx+BF)= \begin{bmatrix}
 0& 0\\
     \widetilde{B}_{N-m}&\widetilde{A}_{N-m}
   \end{bmatrix};\quad \Mx^{-1}BM_{u}= \begin{bmatrix}
   Id_{m}& 0_{m,K-m}\\
   0_{N-m,m}&0_{N-m,K-m}
   \end{bmatrix}.
\ena
This gives the result given the form of $\widetilde{B}_{N-m}$ and $\widetilde{A}_{N-m}$ given by the iteration. The fact that $(\widetilde{A}, \widetilde{B})$ also satisfy the Kalman rank condition follows by direct analysis of the associated control problem for instance.

Finally $\widetilde{A}_{t}= \Mx^{-1}(A\Mx+tBF)=\begin{bmatrix}
  (1-t)T_{1}&(1-t) T_{2}\\
     \widetilde{B}_{N-m}&\widetilde{A}_{N-m}
   \end{bmatrix}$ has the required form.
\enp

\subsection{Eigenvalue problem}
 We will show the following proposition which will be repeatedly used in Section \ref{sectionexample}
\begin{prop}
\lab{eigenproblem:prop}
Assume $A\in \R^{N\times N}$, $B\in \R^{K\times N}$, $\alpha,\beta$ are smooth functions and $\omega=\{\alpha\neq 0\}$, $o=\{\beta\neq 0\}$, respectively. Then, for all $\lambda_1\in\C$, $\lambda_2\in \C\setminus \{0\}$, or $\lambda_1=1, \lambda_2=0$, the eigenvalue problem
\be\lab{eigenproblem:system}
\begin{cases}
(\lambda_1-\Delta_g)U+\lambda_2A\beta U=0,\\
\alpha BU=0.
\end{cases}     \quad\forall x\in \mathcal{M},
\ee
 admits an unique zero solution $U\equiv0$, if $A,B$ and $\alpha,\beta$ satisfy one of the following assumptions
 \begin{enumerate}
\item  $(A^{tr},B^{tr})$ satisfy Kalman Rank Condition and $\tilde{K}=1$ ( number of distinct eigenvalues of $A$ is 1 ), $\beta$ satisfies a sign condition, that is, $\beta\geq0$ (or $\beta\leq0$).
    \item   $(A^{tr},B^{tr})$ satisfy Kalman Rank Condition and $\omega\cap o\neq \emptyset$, $\beta$ satisfies a sign condition, that is, $\beta\geq0$ (or $\beta\leq0$).
\end{enumerate}
\end{prop}
Before we prove Proposition \ref{eigenproblem:prop}, we need to recall some basic facts of linear algebra and state notations related to Jordan decomposition. For any matrix $A\in \R^{N\times N}$, we denote by $\{\mu_{i},i=1,\cdots,\tilde{K}\}$ distinct eigenvalues of $A$. $l_i$ denotes the geometric multiplicity (the dimension of Ker($A-\mu_i)$, that is the number of jordan blocks corresponding to $\mu_i$) of $\mu_i$ for $i=1,\cdots,\tilde{K}$. Let $P_{ij}^1\in \C^N$ be eigenvector corresponding to $\mu_{i}$ for $i=1,\cdots,\tilde{K}; j=1,...,l_i$. We define root vectors $P_{ij}^k\in \C^N$ associated to each eigenvector $\{P_{ij}^1\}$, which are given by
\be
\begin{cases}
(A-\mu_{i})P_{ij}^{k+1}=P_{ij}^{k};1\leq k\leq l_i^j-1\\
(A-\mu_{i})P_{ij}^1=0,
\end{cases}
\ee
where  $l_i^j$ denote the dimension of Jordan chain of $\{P_{ij}^1\}$ for $i=1,\cdots,\tilde{K}; j=1,...,l_i$. Then by classical theory of linear algebra, we can obtain $$\{P_{ij}^k\},i=1,\cdots,\tilde{K}; j=1,...,l_i;k=1,\cdots,l_i^j$$ span a base of
$\C^N$. Define a matrix
\be\displaystyle
P:=\left[P_{11}^1|P_{11}^2|\cdots|P_{11}^{l_1^1}|P_{12}^1|\cdots|P_{1l_1}^{l_1^{l_1}}|P_{21}^1|\cdots
|P_{\tilde{K}l_{\tilde{K}}}^{l_{\tilde{K}}^{l_{\tilde{K}}}}\right],
\ee
so we have Jordan Canonical Form $\tilde{A}$ of $A$:
\be\displaystyle\lab{jcf:a}
\tilde{A}:=P^{-1}AP=diag(A_1,A_2,\cdots,A_{\tilde{K}})
\ee
where
\be\lab{jcf:a_{ij}} \small
\displaystyle
A_{i}=diag(A_{i1},\cdots,A_{il_i}),
\ee
and $A_{ij}$ is $C^{l_i^j\times l_i^j}$ jordan block corresponding to $\mu_i$
for $i=1,\cdots,\tilde{K} ; j=1,\cdots,l_i$.

Let $
\tilde{B}:=BP,
$
then we state the proof of proposition \ref{eigenproblem:prop}.
\bnp[Proof of Proposition \ref{eigenproblem:prop}]
Case  ''$\lambda_1=1,\lambda_2=0 $" is simple, since $1-\Delta$ is a positive operator, then $U=0$. So we only need to prove case
'$\lambda_2 \neq 0$'.
 Let $W:=P^{-1}U$. Since $\tilde{A}$ satisfies \eqref{jcf:a}, System \eqref{eigenproblem:system} can be decoupled of $\tilde{K}$  blocks, so that we only need to consider the solution $W_{ij}=(W_{ij}^1,\cdots,W_{ij}^{l_i^j})\in (C^\infty(\mathcal{M}))^{l_i^j}$ of the following problem:
\be
\lab{eigenproblem:subsystem}
(\lambda_1-\Delta_g)W_{ij}+\lambda_2A_{ij}\beta W_{ij}=0
, \quad \forall x\in \mathcal{M},
\ee
where $A_{ij}$ is given by \eqref{jcf:a_{ij}} for every $i=1,\cdots,\tilde{K} ; j=1,\cdots,l_i$.
More precisely, we rewrite System \eqref{eigenproblem:subsystem} as follow,
\be
\begin{cases}
(\lambda_1-\Delta_g)W_{ij}^1+\lambda_2\mu_{i}\beta W_{ij}^1+{\lambda_2\beta} W_{ij}^2=0,\\
\vdots\\
(\lambda_1-\Delta_g)W_{ij}^{l_i^j}+\lambda_2\mu_{i}\beta W_{ij}^{l_i^j}=0.
\end{cases}
\ee
Multiplying $W_{ij}^{l_i^j-1}$-equation by $\bar{W}_{ij}^{l_i^j}$ and by integration by parts over $\mathcal{M}$, since $\beta$ satisfies sign condition, we have
\be
\nonumber
W_{ij}^{l_i^j}=0, \quad\forall x\in o.
\ee
Then by unique continuation of scalar elliptic equation, we obtain
\be
\nonumber
W_{ij}^{l_i^j}=0, \quad \forall x\in \mathcal{M}.
\ee
Hence, repeating this process to each equation
of $\{W_{ij}^k\},$ for $k=2,\cdots,l_i^j$, we obtain
\be
\nonumber
W_{ij}^k=0, \quad\forall k=2,\cdots,l_i^j, x\in \mathcal{M}.
\ee
It suffices to show that $W_{ij}^1=0, i=1,\cdots,\tilde{K},j=1,\cdots,l_i$ under assumptions 1 or 2. Indeed, $W_{ij}^1$ satisfies the following equation
\be\lab{eigenvalue prob:system fini}
(\lambda_1-\Delta_g)W_{ij}^1+\lambda_2\mu_i\beta W_{ij}^1=0, \quad i=1,\cdots,\tilde{K},j=1,\cdots,l_i.
\ee
 Since $\tilde{B}$ can be rewritten as
\be
\left[BP_{11}^1|\cdots|BP_{\tilde{K}l_{\tilde{K}}}^{l_{\tilde{K}}^{l_{\tilde{K}}}} \right],
\ee
then
\be\lab{eigenvalue prob:observ}
\tilde{B}W \alpha =\alpha\sum_{i,j} BP_{ij}^1W_{ij}^1=\alpha \sum^{\tilde{K}}_{i=1} (\sum^{l_i}_{j=1} BP_{ij}^1W_{ij}^1)=0.
\ee
If we have assumption 1, that is, $\tilde{K}=1$ and $A,B$ satisfy Kalman rank condition. Then we obtain that for $j=1,\cdots,l_1, x\in \omega$,  $W_{1j}^1=0$ following from \cite[Proposition 3.1]{AFAMD:11}. By unique continuation of scalar elliptic equation \eqref{eigenvalue prob:system fini}, we have $W_{1j}^1=0,x\in \mathcal{M},\forall j=1,\cdots,l_1$.

Next, if we have Assumption 2, that is,  $A,B$ satisfy Kalman rank condition and $\omega\cap o\neq \emptyset$, then set $\tilde{\omega}\subset \omega\cap o$,  in view of \eqref{eigenvalue prob:observ}, we have
\be
(\lambda_1-\Delta_g)\sum_{i} (\sum_j BP_{ij}^1W_{i1}^1)=0, \quad \forall x\in \tilde{\omega}.
 \ee
By using \eqref{eigenvalue prob:system fini}, we have
\be
  \beta\sum_{i}\lambda_2\mu_i (\sum_j BP_{ij}^1W_{ij}^1)=0,  \quad\forall x\in \tilde{\omega}.
\ee
By induction, we obtain
\be
 \beta\sum^{\tilde{K}}_{i=1}(\lambda_2\mu_i )^k (\sum_{j=1}^{l_i} BP_{ij}^1W_{ij}^1)=0, \quad \forall x\in \tilde{\omega}, k=1,\cdots,\tilde{K}.
\ee
Since $\{\mu_i\}_{1,\cdots,\tilde{K}}$ are different, we have
\be
\sum_{j=1}^{l_i} BP_{ij}^1W_{ij}^1=0, \quad \forall x\in \tilde{\omega}, i=1,\cdots,\tilde{K}.
\ee
By \cite[Proposition 3.1]{AFAMD:11},
we obtain
\be
W_{ij}^1=0,  \quad \forall x\in \tilde{\omega}, i=1,\cdots,\tilde{K},j=1,\cdots,l_i.
\ee
By \eqref{eigenvalue prob:system fini} and unique continuation of scalar elliptic equation, we have
\be
W_{ij}^1=0,   \quad
\forall  x\in \mathcal{M}, i=1,\cdots,\tilde{K},j=1,\cdots,l_i.
\ee
 So we finish the proof.
\enp

\subsection{Proof of Lemma \ref{lemma: Regularity}}\lab{1smooth}

In the main part of the paper, we use a matrix operator type version of 1-smooth effect Lemma \ref{lemma: Regularity}. A version of such a result in scalar case can be found in \cite{LaurentLeautaud16}. The Proof of Lemma \ref{lemma: Regularity} relies on the following lemma.
\begin{lem}
\label{l:comm-evol-lambda}
Let $\mathcal{I} \subset \R$ be an interval and let $H_\pm(t)=\pm\Lambda Id_{N\times N}+iW_0(t), W_0\in C^\infty(\mathcal{I};\Psi^0_{phg} (\mathcal{M};\C^{N\times N}))$. Define
 $S_\pm(t,0)$ as the solution operator for the evolution equation $\pa_t - i H_{\pm}(t)$ respectively. Then, for any $A \in  \Psi^m_{phg} (\mathcal{M};\C^{N\times N})$, we have
\be
\label{e:commutator-group}
[A  ,  S_\pm(t,0)] = \int_0^t S_\pm(t,s) [A ,i H_\pm(s)] S_\pm(s,0) ds.
\ee
 In particular, if we take $A = \Lambda Id_{N\times N}$, then, for all $s \in \R$, we have
  \be
  [\Lambda ,  S_\pm(t,0)] , [\Lambda ,  S_\pm(t,0)^*] \in \mathcal{B}_{loc}(\mathcal{I} ;\mathcal{L}(H^{s}(\mathcal{M};\C^{N\times N}))).\ee
\end{lem}

\bnp[Proof of Lemma \ref{l:comm-evol-lambda}]
 Let
 \be
 u_\pm(t) = [A  ,  S_\pm(t,0)] u_0 = A S_\pm(t, 0)u_0 - S_\pm(t,0)A u_0,~~u_{\pm}(0) =0.
 \ee
 solves
\beq\begin{split}
\pa_t u_\pm(t) = A i H_\pm(t) S_\pm(t, 0)u_0 - i H_\pm(t) S_\pm(t,0)A u_0 = [A, iH_\pm(t)]S_\pm(t,0) u_0  + i H_\pm(t) u_\pm(t) .
\end{split}\eeq
so that the  Duhamel principal yields \eqref{e:commutator-group}. We finish the proof of Lemma \ref{l:comm-evol-lambda}.\enp

\bnp[Proof of Lemma \ref{lemma: Regularity}] 
We refer for instance to \cite[Section A.3]{LaurentLeautaud16} for some details in the scalar case, the proof being almost the same. So we omit it.
\enp
\def\cprime{$'$}

\end{CJK*}

\end{document}